\documentclass[11pt]{article}
\usepackage{amsmath,amsthm, amssymb, amsfonts,accents,nccmath}
\usepackage{enumitem}
\usepackage{array}
\usepackage{lipsum}
\usepackage{mathtools}
\usepackage[toc,page]{appendix}
\usepackage[english]{babel}
\usepackage{dsfont}
\usepackage{booktabs}
\usepackage[linesnumbered,commentsnumbered,ruled,vlined]{algorithm2e}
\usepackage{tikz-network}
\usetikzlibrary{hobby}
\usepackage[utf8]{inputenc}
\usepackage[english]{babel}

\usepackage{caption}
\usepackage{subcaption}

\usepackage{cite}

\newcommand\restr[2]{{% we make the whole thing an ordinary symbol
  \left.\kern-\nulldelimiterspace % automatically resize the bar with \right
  #1 % the function
  \vphantom{\big|} % pretend it's a little taller at normal size
  \right|_{#2} % this is the delimiter
  }}

\usepackage{hyperref}
\hypersetup{
    colorlinks=true,
    linkcolor=blue,
    filecolor=darkmagenta,      
    urlcolor=cyan,
    citecolor=red
} 
\makeatletter
\def\BState{\State\hskip-\ALG@thistlm}
\makeatother
\numberwithin{equation}{section}

\newtheorem{theorem}{Theorem}[section]
\newtheorem{cor}[theorem]{Corollary}
\newtheorem{lem}[theorem]{Lemma}
\newtheorem{prop}[theorem]{Proposition}
\newtheorem{defn}[theorem]{Definition}
\newtheorem{rem}[theorem]{Remark}

\newcommand{\ds}{\displaystyle}

\title{On the spectral radius of block graphs with a given \\ dissociation number }
\author{Joyentanuj Das\footnote{Department of Applied Mathematics,
National Sun Yat-sen University, Gushan District, Kaohsiung City, 804, Taiwan (ROC) \indent  Email: joyentanuj@gmail.com,  joyentanuj@math.nsysu.edu.tw}  \quad and \quad Sumit Mohanty\footnote{Humanities and Applied Sciences, IIM Ranchi,  Prabandhan Nagar, Vill-Mudma, Nayasarai Road, Ranchi, Jharkhand-834003, India. \indent   Email:  sumitmath@gmail.com, sumit.mohanty@iimranchi.ac.in}}

\date{}

\begin{document}

\maketitle

\begin{abstract}
A connected graph is called a block graph if each of its blocks is a complete graph. Let $\mathbf{Bl}(\textbf{k}, \varphi)$  be the class of block graph on $\textbf{k}$ vertices with given dissociation number $\varphi$. In this article, we have obtained a block graph $\mathbb{B}_{\textbf{k},\varphi}$ in  $\mathbf{Bl}(\textbf{k}, \varphi)$ that uniquely attains the maximum spectral radius $\rho(G)$ among all graphs $G$ in  $\mathbf{Bl}(\textbf{k}, \varphi)$. Furthermore, we also provide bounds on  $\rho(\mathbb{B}_{\textbf{k},\varphi})$.
\end{abstract}



\section{Introduction}\label{sec:intro}
Let $G=(V(G),E(G))$  be a finite, simple, connected graph with $V(G)$ as the set of vertices and $E(G)$ as the set of edges in $G$.  We write $u\sim v$ to indicate that the vertices $u,v \in V(G)$ are adjacent in $G$.   The degree of the vertex $v$, denoted by $d_G(v)$, equals the number of vertices in $V$ that are adjacent to $v$.   A graph $H$ is said to be a subgraph of $G$ if $V(H) \subset V(G)$ and $E(H) \subset E(G)$. For any subset $S \subset V (G)$, a subgraph $H$ of $G$ is said to be an induced subgraph with vertex set $S$, if $H$ is a maximal subgraph of $G$ with vertex set $V(H)=S$. We write $|S|$ to denote the cardinality of the set $S$. 

Let $G=(V(G),E(G))$ be a graph. For $u,v \in V$, the adjacency matrix of the graph  $G$ is,    $\mathbf{A}(G) = [a_{uv}]$, where $a_{uv}= 1 $  if $u\sim v$ and $0$ otherwise.  For any column vector $\mathbf{x}$, if $x_u$  represent the entry of $\mathbf{x}$ corresponding to vertex $u\in V$, then 
$$\mathbf{x}^T \mathbf{A}(G)\mathbf{x}=  2\sum_{u\sim w}x_{u}x_{w}, $$
where $\mathbf{x}^T$ represents  the transpose of  $\mathbf{x}$. 

%%Similarly, if $(A\mathbf{x})_u$ represent the entry of $A\mathbf{x}$ corresponding to vertex $u\in V$, then  $$(A\mathbf{x})_u=\sum_{w\sim u}x_{w}.$$

For a connected graph $G$    on $n \geq 2$ vertices, by the Perron-Frobenius theorem, the spectral radius  $\rho(G)$ of $\mathbf{A}(G)$ is a simple positive eigenvalue and the associated eigenvector is entry-wise positive (for details see~\cite{Bapat}). We will refer to such an eigenvector as a Perron vector of $G$.  Now we state a few known results on spectral radius useful in our subsequent calculations. By the Min-max theorem, we have
\begin{equation}\label{eqn:eq_sp_rd}
\rho(G) = \max_{\mathbf{x}\neq 0}\dfrac{\mathbf{x}^T \mathbf{A}(G)\mathbf{x}}{\mathbf{x}^T \mathbf{x}}= \max_{\mathbf{x}\neq 0}\dfrac{2\sum_{u\sim w}x_{u}x_{w}}{\sum_{u\in V}x_{u}^2}.
\end{equation}
Furthermore, in Eqn.~\eqref{eqn:eq_sp_rd}, the maximum attained if and only if $\mathbf{x}$ is a Perron vector of $G.$

Given a graph $G=(V(G),E(G))$, for $u, v \in V(G)$ we will use  $G + uv$ to denote the graphs obtained from $G$  by adding an edge $uv \notin E(G)$ and we have the following result.
\begin{lem}\label{lem:sr_edge}\cite{Bapat, Li}
	If $G$ is a graph such that for $u, v \in V(G)$, $uv \notin E(G)$, then $\rho(G) < \rho(G+uv)$. 
\end{lem}

A vertex $v$ of a connected graph $G$ is a cut vertex of $G$ if $G - v$ is disconnected. A block of the graph $G$ is a maximal connected subgraph of G that has no cut-vertex.  A pendant block is a block whose deletion does not disconnect the graph.  Given two blocks  $F$ and $H$  of graph $G$ are said to be adjacent if they are connected via a cut-vertex. We denote $F\circledcirc H$, to represent the induced subgraph on the vertex set of two adjacent blocks $F$ and $H$.

Let $G$ be a graph.  The neighbours of a vertex $v$ in $G$ is the set of vertices that are adjacent to $v$ and denoted by $N_G(v)$. Similarly, given a cut vertex $v$ in $G$,  we denote $N_{G}^{\mathbf{Bl}}(v)$ is the set of blocks in $G$ that are adjacent to each other via the cut vertex $v$.

 A complete graph is a graph where each vertex is adjacent to every other vertex. A complete graph on $n$ vertices is denoted by $K_n$. A connected graph is called a block graph if each of its blocks is a complete graph.  Let $G$ be a block graph with $d_1$ blocks of $K_{n_1}$, $d_2$ blocks of $K_{n_2}$, so on up to $d_b$ blocks of $K_{n_b}$, then we write $G$  with blocks $ K_{n_1}^{(d_1)} - K_{n_2}^{(d_2)}- \cdots - K_{n_b}^{(d_b)}$ (for example see Figure~\ref{fig:block}). Here the above notation only gives information only about the blocks of $G$, but not about the structure of the graph. But for a special case, if $G$ has a central cut vertex, then all the blocks are adjacent via the central cut vertex and we denote it by
$$G = K_{n_1}^{(d_1)} \circledcirc K_{n_2}^{(d_2)}\circledcirc \cdots \circledcirc K_{n_b}^{(d_b)}.$$

\begin{figure}[ht]
	\centering
	\begin{tikzpicture}[scale=.6]
		\Vertex[size=.1,color=red]{A}
		\Vertex[y=2,size=.1,color=black]{B}
		\Vertex[x=2,y=0,size=.1,color=red]{C} 
		\Vertex[x=2,y=2,size=.1,color=black]{D}
		\Vertex[x=2,y=-2,size=.1,color=black]{E}
		\Vertex[x=4,y=0,size=.1,color=black]{F} 
		\Vertex[x=4,y=-2,size=.1,color=black]{G}
		
		\Vertex[x=3,y=2,size=.1,color=black]{D1}
		\Vertex[x=4,y=1,size=.1,color=black]{F1}

		\Vertex[x=-2,y=-.7,size=.1,color=black]{D2}
		\Vertex[x=-2,y=.7,size=.1,color=black]{F2}	
		
		\Vertex[x=-1,y=-2,size=.1,color=black]{D3}
		\Vertex[x=1,y=-2,size=.1,color=black]{E3}	
		
		\Edge[lw=2pt](A)(B)
		\Edge[lw=2pt](A)(C)
		\Edge[lw=2pt](A)(D)
		\Edge[lw=2pt](B)(C)
		\Edge[lw=2pt](B)(D)
		\Edge[lw=2pt](C)(D)
		
		\Edge[lw=2pt](C)(E)
		\Edge[lw=2pt](C)(F)
		\Edge[lw=2pt](C)(G)
		\Edge[lw=2pt](E)(F)
		\Edge[lw=2pt](E)(G)
		\Edge[lw=2pt](F)(G)
		
		\Edge[lw=2pt](C)(D1)
		\Edge[lw=2pt](C)(F1)
		\Edge[lw=2pt](D1)(F1)
		
		\Edge[lw=2pt](A)(D2)
		\Edge[lw=2pt](A)(F2)
		\Edge[lw=2pt](D2)(F2)
		
		\Edge[lw=2pt](A)(D3)
		\Edge[lw=2pt](A)(E3)
	%	\draw[red, thick, dashed, use Hobby shortcut, closed]  (C)..(2.3,2.5)..(0,2.8)..(-3,0)..(1,-2.5);
	\end{tikzpicture}
\caption{A block graph with blocks  $K_{1}^{(2)}-K_{3}^{(2)}-K_{4}^{(2)}$.} \label{fig:block}
\end{figure}
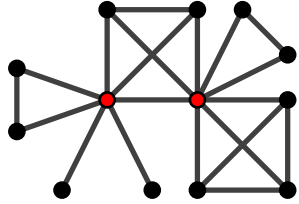

A vertex cover of a graph is a set of vertices that includes at least one endpoint of every edge of the graph. Given a graph $G$ and a positive integer $k$, a subset $S$ of the vertices of $G$ is a vertex $k$-path cover if, for every path on $k$ vertices in $G$, it contains a vertex from $S$. The minimum cardinality of a vertex $k$-path cover is denoted by $\psi_k(G)$. The graph invariant $\psi_k(G)$, was introduced in~\cite{Bresar,Bresar1} motivated by the problem of ensuring data integrity of communications in wireless sensor networks. One can see that $\psi_2(G)$ is the same as the minimum cardinality of a vertex cover of $G$; hence, the concept of vertex $k$-path cover is a generalization of vertex cover.

A set $\mathcal{I}$ of vertices in a graph $G$ is an independent set if no pair of vertices of $\mathcal{I}$ are adjacent. The independence number of $G$, denote by $\alpha(G)$, is the maximum cardinality of an independent set in $G$. We have a nice relation between $\psi_2(G)$ and $\alpha(G)$ as follows: $$\psi_2(G) + \alpha(G) = |V(G)|.$$

A set $\mathcal{D}$ of vertices in a graph $G$ is a dissociation set if the induced subgraph with vertex set $\mathcal{D}$ has maximum degree at most $1$. A maximum dissociation set of $G$ is a dissociation set with maximum cardinality. The dissociation number of $G$ is denoted by $\varphi(G)$ is the cardinality of a maximum dissociation set. Note that $\mathcal{I} \subset \mathcal{D}$. Clearly one can also observe that $\psi_3(G)$ and $\varphi(G)$ are related as follows: $$\psi_3(G) + \varphi(G) = |V(G)|.$$
We now present a few observations following directly from the above definitions.  Let $G$ be a block graph and $\mathcal{D}$ be a maximal dissociation set of $G$. Then, the following holds.
\begin{itemize}
\item $\varphi(G)=2$    if and only if  either $G$ is $K_n$ for some $n\geq 2$  or a path of length $2$.

\item If $B$ is a block in $G$, then $|B \cap \mathcal{D}|\leq 2.$ 

\item If $B$ is a pendant block in $G$, then $|B \cap \mathcal{D}|=2$.

\end{itemize}

The spectral radius of a graph  has been extensively studied subject to various graph-theoretic constraints. In spectral graph theory, the maximization and minimization of the spectral radius of a given class of a graph and the problem of determining the extremal graphs find a particular interest among researchers. One such result due to Brualdi and Solheid in~\cite{Brualdi} is an important motivation for our study. In this article, the authors obtained the graphs that maximize the spectral radius amongst graphs with a fixed number of vertices $\mathbf{k}$ in a given class of graphs. Later, several results in similar contexts has been published (for example see~\cite{Conde,Das, Feng,Guo,Chu, Lui,Lu, Lou, Xiao, Xu}). Obtaining bounds on the spectral radius of a graph is also an important direction in spectral graph theory, and for a few of the interesting results on bounds, readers may refer~\cite{KCD,Yuan}. For a few more interesting results on spectral graph theory,  readers also may refer~\cite{DS}.

In~\cite{Chu}, the authors found the graph that uniquely maximizes the spectral radius amongst the trees on a fixed number of vertices $\mathbf{k}$ with a given independence number. Subsequently, characterizing extremal graphs with a given independence number were studied for various classes of graphs (for details see~\cite{Conde, Das, Feng, Xu}). In particular, block graphs on a fixed number of vertices $\mathbf{k}$ and with a given independence number have been studied in~\cite{Conde}. Further, in~\cite{Das}, bi-block graphs (each of its blocks is a complete bipartite graph) have been studied, and later in~\cite{Lou}  the results of bi-block graphs were extended for general bipartite graphs.

From the definition of the dissociation number, it is clear that the dissociation number is an immediate generalization of the independence number. Several interesting results have been published related to dissociation number (for example see~\cite{Orlovich, Tu, Tu1}). Motivated by the characterization of extremal graphs concerning the spectral radius and independence number,  we have obtained a graph that uniquely maximizes the spectral radius amongst the class of block graphs with a fixed number of vertices $\mathbf{k}$ and a given dissociation number $\varphi$.  We now present our main result. 

Let $\mathbf{Bl}(\textbf{k}, \varphi)$ be the class of block graph on $\textbf{k}$ vertices with given dissociation number $\varphi$. 	 We now define a special block graph with a central cut vertex in $\mathbf{Bl}(\textbf{k}, \varphi)$. Let 
	\begin{align*}
		\mathbb{B}_{\textbf{k},\varphi} =
		\begin{cases}
			K_2 \circledcirc K_3^{(\frac{\varphi-3}{2})}\circledcirc K_{\textbf{k}-\varphi+2}, & \text{ if } \varphi \text{ is odd,}\\
			\\
						K_3^{(\frac{\varphi-2}{2})}\circledcirc K_{\textbf{k}-\varphi+2}, & \text{ if } \varphi \text{ is even.}
		\end{cases}
	\end{align*}

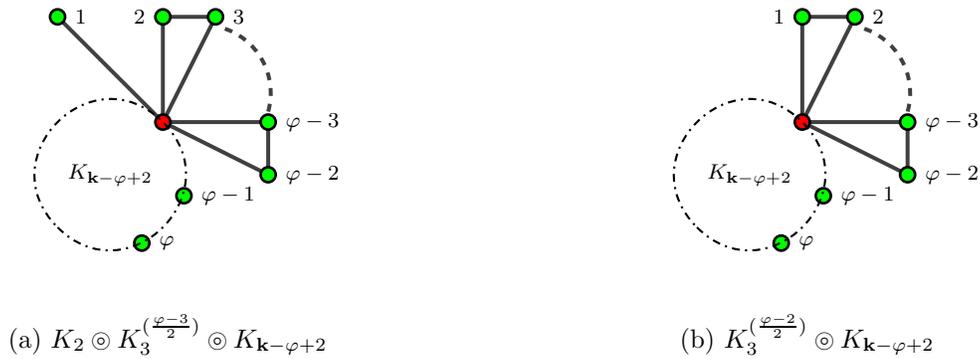
\begin{figure}[ht]
	\centering
	\begin{subfigure}[b]{0.5\linewidth}
		\centering
		\begin{tikzpicture}[scale=.7]
			\Vertex[x=-1,y=-1,label=$K_{\textbf{k}-\varphi+2}$,style={shape=coordinate}]{-1}
			\Vertex[x=-2,y=-2,style={shape=coordinate}]{-2}
			\Vertex[x=0,y=0,color=red,size=.2]{0}
			\Vertex[x=0,y=2,label=$2$,position=left,color=green,size=.2]{1}
			\Vertex[x=1,y=2,label=$3$,position=right,color=green,size=.2]{2}
			\Vertex[x=2,y=0,label=$\varphi-3$,position=right,color=green,size=.2]{3}
			\Vertex[x=2,y=-1,label=$\varphi-2$,position=right,color=green,size=.2]{4}
			\Vertex[x=-2,y=2,label=$1$,position=right,color=green,size=.2]{5}
			\Vertex[x=1.1,y=1.8,style={shape=coordinate}]{6}
			\Vertex[x=2,y=0.2,style={shape=coordinate}]{7}
			
			\Vertex[x=0.4,y=-1.4,label=$\varphi-1$,position=right,color=green,size=.2]{8}
			\Vertex[x=-0.4,y=-2.3,label=$\varphi$,position=right,color=green,size=.2]{9}
			
			%\Edge[bend=80](-2)(0)
			%\Edge[bend=-90](-2)(0)
			\Edge(0)(1)
			\Edge(0)(2)
			\Edge(1)(2)
			\Edge(3)(0)
			\Edge(0)(4)
			\Edge(4)(3)
			\Edge(0)(5)
			\Edge[style={dashed},bend=-45](7)(6)
			
			\draw[black, thick, dash dot, use Hobby shortcut,closed](C)(-2,-2)..(0,0);
		\end{tikzpicture}
		\caption{$K_2 \circledcirc K_3^{(\frac{\varphi-3}{2})}\circledcirc K_{\textbf{k}-\varphi+2}$}
	\end{subfigure}%
	\begin{subfigure}[b]{0.5\linewidth}
		\centering
		\begin{tikzpicture}[scale=.7]
			\Vertex[x=-1,y=-1,label=$K_{\textbf{k}-\varphi+2}$,style={shape=coordinate}]{-1}
			\Vertex[x=-2,y=-2,style={shape=coordinate}]{-2}
			\Vertex[x=0,y=0,color=red,size=.2]{0}
			\Vertex[x=0,y=2,label=$1$,position=left,color=green,size=.2]{1}
			\Vertex[x=1,y=2,label=$2$,position=right,color=green,size=.2]{2}
			\Vertex[x=2,y=0,label=$\varphi-3$,position=right,color=green,size=.2]{3}
			\Vertex[x=2,y=-1,label=$\varphi-2$,position=right,color=green,size=.2]{4}
			%%%\Vertex[x=-2,y=2,label=$1$,position=right,color=black,size=.2]{5}
			\Vertex[x=1.1,y=1.8,style={shape=coordinate}]{6}
			\Vertex[x=2,y=0.2,style={shape=coordinate}]{7}
			
			\Vertex[x=0.4,y=-1.4,label=$\varphi-1$,position=right,color=green,size=.2]{8}
			\Vertex[x=-0.4,y=-2.3,label=$\varphi$,position=right,color=green,size=.2]{9}
			
			%\Edge[bend=80](-2)(0)
			%\Edge[bend=-90](-2)(0)
			\Edge(0)(1)
			\Edge(0)(2)
			\Edge(1)(2)
			\Edge(3)(0)
			\Edge(0)(4)
			\Edge(4)(3)
			%\Edge(0)(5)
			\Edge[style={dashed},bend=-45](7)(6)
			
			\draw[black, thick, dash dot, use Hobby shortcut,closed](C)(-2,-2)..(0,0);
		\end{tikzpicture}
		\caption{$K_3^{(\frac{\varphi-2}{2})} \circledcirc K_{\textbf{k}-\varphi+2}$}
	\end{subfigure}
	\caption{ $\mathbb{B}_{\textbf{k},\varphi}$, extremal block graphs with a given $\textbf{k}$ and $\varphi$.}\label{fig:extgraph}
\end{figure}
	
It is easy to see that  $\mathbb{B}_{\textbf{k},\varphi}$ is in $\mathbf{Bl}(\textbf{k}, \varphi)$ (see Figure~\ref{fig:extgraph}). In this article, we prove that the maximum spectral radius $\rho(G)$  amongst all graphs $G$ in $\mathbf{Bl}(\textbf{k}, \varphi)$  is uniquely attained by $\mathbb{B}_{\textbf{k},\varphi}$. 

This article is organized as follows. In Section~\ref{sec:prelim}, we have proved a few preliminary results that are necessary for the subsequent sections. In Section~\ref{sec:maxi}, we first build up the necessary tools in the form of a few results and used these results to prove the maximum spectral radius $\rho(G)$  amongst all graphs $G$ in $\mathbf{Bl}(\textbf{k}, \varphi)$  is uniquely attained by $\mathbb{B}_{\textbf{k},\varphi}$. Finally, in Section~\ref{sec:bounds}, we obtained bounds on $\rho(\mathbb{B}_{\textbf{k},\varphi})$.

\section{Preliminary Results}\label{sec:prelim}

 We begin with notions of equitable partition and the equitable quotient matrix. The  equitable quotient matrix is considered one of the powerful tools to study of spectrum of a matrix. Although, the equitable quotient matrix is defined for more general set-up, but we restrict ourself to the result that is relevant to this article.  For more general set-up reader may refer~\cite{You}.

Suppose $M$ is a symmetric real matrix whose rows and columns are indexed by
$X = \{1, . . . , n \}$. Let $\{X_1, . . . ,X_t \}$ be a partition of $X$.  Let $M$ be
partitioned according to $\{X_1, . . . ,X_t \}$, \emph{i.e.},
\begin{equation*}\label{eqn:M}
		M = \begin{bmatrix}
			M_{11}       &  \dots & M_{1t} \\
			\vdots & \ddots & \vdots\\
			M_{t1}       &  \dots & M_{tt}
		\end{bmatrix},
	\end{equation*}
where $M_{ij}$ denotes the submatrix (block) of $M$ formed by rows in $X_i$ and the
columns in $X_j$. Let $q_{ij}(M)$ denote the average row sum of $M_{ij}$. 
Then the matrix $\mathbf{Q}(M) = [q_{ij}(M)]$ is called the quotient matrix of $M $ with respect to given partition. 	If the row sum of each block $M_{ij}$ is constant then the partition is called equitable, and $\mathbf{Q}(M)$ is called the equitable quotient matrix of $M$.

We now state a theorem that gives a relation between the spectrum of a matrix and the spectrum of its quotient matrix. 
	
\begin{theorem}\label{thm:quesient}~\cite{You}
Let $M$ be a symmetric real matrix such that  a quotient matrix $\mathbf{Q}(M)$ of $M$ is  equitable. Then $\sigma (\mathbf{Q}(M)) \subset \sigma (M)$. Furthermore, if $M$ is a nonnegative matrix, then $\rho (M)= \rho(\mathbf{Q}(M))$. 
\end{theorem}

We now introduce some notations. Let $I_n$ and $ \mathds{1}_n$  denote the identity matrix and the column vector of all ones, respectively. Further, $J_{m \times n}$  denotes the $m\times n$ matrix of all ones and if $m=n$, we use the notation $J_m$. We will use $\mathbf{0}_{m \times n}$ to represent zero matrix. We now prove interesting result on the spectral radius  $\rho(\mathbb{B}_{\textbf{k},\varphi})$.

\begin{lem}\label{lem:roots}
The spectral radius  $\rho(\mathbb{B}_{\textbf{k},\varphi})$ of  $\mathbb{B}_{\textbf{k},\varphi}$ is the largest root of a polynomial 
\begin{equation}\label{eqn:f_g}
\begin{cases}
f(x)=x^3-(\textbf{k}-\varphi+1)x^2-(\varphi-1)x+(\varphi-1)(\textbf{k}-\varphi)+1   & \mbox{ if } \varphi \mbox{ is even,}\\
\\

g(x)=x^4 -(\textbf{k}-\varphi+1)x^3-(\varphi-1)x^2+ ((\varphi-1)(\textbf{k}-\varphi)+2)x - (\textbf{k}-\varphi) & \mbox{ if } \varphi \mbox{ is odd}.
\end{cases}
\end{equation}
\end{lem}

\begin{proof}
	Let $\varphi$ be even and  $\varphi = 2s$ for some $s$. Then, the adjacency matrix of $\mathbb{B}_{\textbf{k},\varphi}$, can be expressed in the following block form
	\begin{equation}\label{eqn:A(B_n,k)}
		\mathbf{A}(\mathbb{B}_{\textbf{k},\varphi}) = \begin{bmatrix}
			J_{\textbf{k}-2s+1} -I_{\textbf{k}-2s+1} & \mathds{1}_{\textbf{k}-2s+1} &\textbf{0}_{(\textbf{k}-2s+1)\times(2s-2)}\\
			\mathds{1}_{\textbf{k}-2s+1}^T & 0 & \mathds{1}_{2(s-1)}^T\\
			\textbf{0}_{(\textbf{k}-2s+1)\times(2s-2)}^T & \mathds{1}_{2(s-1)} & \mathbf{C}
		\end{bmatrix},
	\end{equation} where $\mathbf{C}$ is a block diagonal matrix of order $2(s-1)$ with $J_2-I_2$ as the $2 \times 2$ diagonal blocks. Observe that, each block has constant row sum, hence the equitable quotient matrix of $\mathbf{A}(\mathbb{B}_{\textbf{k},\varphi})$ is given by
	\begin{equation*}\label{eqn:B(A)}
		\mathbf{Q}(\mathbf{A}(\mathbb{B}_{\textbf{k},\varphi})) = \begin{bmatrix}
			\textbf{k}-2s & 1 & 0\\
			\textbf{k}-2s+1 & 0 & 2(s-1)\\
			0 & 1 & 1
		\end{bmatrix}
	\end{equation*} 
and the characteristic polynomial of $\mathbf{Q}(\mathbf{A}(\mathbb{B}_{\textbf{k},\varphi}))$ is given by 
\begin{equation}\label{eqn:ch_f}
f(x)=\det(xI - \mathbf{Q}(\mathbf{A}(\mathbb{B}_{\textbf{k},\varphi}))) =  x^3-(\textbf{k}-\varphi+1)x^2-(\varphi-1)x+(\varphi-1)(\textbf{k}-\varphi)+1.
\end{equation}
%%$$f(x)=\det(xI - \mathbf{Q}(\mathbf{A}(\mathbb{B}_{\textbf{k},\varphi}))) =  x^3-(\textbf{k}-\varphi+1)x^2-(\varphi-1)x+(\varphi-1)(\textbf{k}-\varphi)+1.$$ 
	
		For the case when $\varphi$ is odd, let us assume that $\varphi = 2s+1$ for some $s$. Then, the adjacency matrix of $\mathbb{B}_{\textbf{k},\varphi}$ can be expressed in the following block form
	\begin{equation*}\label{eqn:A(B_n,k)-odd}
		\mathbf{A}(\mathbb{B}_{\textbf{k},\varphi}) = \begin{bmatrix}
			J_{\textbf{k}-2s} -I_{\textbf{k}-2s} & \mathds{1}_{\textbf{k}-2s} &\textbf{0}_{(\textbf{k}-2s)\times(2s-2)} & \mathbf{0}_{(\textbf{k}-2s)}\\
			\mathds{1}_{\textbf{k}-2s}^T & 0 & \mathds{1}_{2(s-1)}^T & 1\\
			\textbf{0}_{(\textbf{k}-2s)\times(2s-2)}^T & \mathds{1}_{2(s-1)} & \mathbf{C} & \mathbf{0}_{2(s-1)}\\
			\mathbf{0}_{(\textbf{k}-2s)}^T & 1 & \mathbf{0}_{2(s-1)}^T & 0
		\end{bmatrix},
	\end{equation*} where $\mathbf{C}$ is a block diagonal matrix of order $2(s-1)$ with $J_2-I_2$ as the $2 \times 2$ diagonal blocks. Observe that, each block has constant row sum, hence the equitable quotient matrix of $\mathbf{A}(\mathbb{B}_{\textbf{k},\varphi})$ is given by
\begin{equation*}\label{eqn:B(A)-odd}
\mathbf{Q}(\mathbf{A}(\mathbb{B}_{\textbf{k},\varphi})) = \begin{bmatrix}
	\textbf{k}-2s-1 & 1 & 0 & 0\\
	\textbf{k}-2s & 0 & 2(s-1) & 1\\
	0 & 1 & 1 & 0\\
	0 & 1 & 0 & 0
\end{bmatrix}
\end{equation*} and the characteristic polynomial of $\mathbf{Q}(\mathbf{A}(\mathbb{B}_{\textbf{k},\varphi}))$ is given by 
\begin{equation}\label{eqn:ch_g}
\det(xI - \mathbf{Q}(\mathbf{A}(\mathbb{B}_{\textbf{k},\varphi}))) = x^4 -(\textbf{k}-\varphi+1)x^3-(\varphi-1)x^2+ ((\varphi-1)(\textbf{k}-\varphi)+2)x - (\textbf{k}-\varphi).
\end{equation}
%%$$\det(xI - \mathbf{Q}(\mathbf{A}(\mathbb{B}_{\textbf{k},\varphi}))) = x^4 -(\textbf{k}-\varphi+1)x^3-(\varphi-1)x^2+ ((\varphi-1)(\textbf{k}-\varphi)+2)x - (\textbf{k}-\varphi).$$
Therefore, in view of Eqns.~\eqref{eqn:ch_f} and~\eqref{eqn:ch_g}, the result follows from Theorem~\ref{thm:quesient}.
\end{proof}

Given Lemma~\ref{lem:roots}, we are ready to compute the characteristic polynomial of  $\mathbf{A}(\mathbb{B}_{\textbf{k},\varphi}).$

\begin{theorem}
For $2 \le \varphi \le \textbf{k} -1$, the  characteristic polynomial $\chi_{\mathbf{A}}(x)$ of  the adjacency matrix $\mathbf{A}(\mathbb{B}_{\textbf{k},\varphi})$ is given by

	\begin{equation*}
		\chi_{\mathbf{A}}(x) = \det(xI - \mathbf{A}(\mathbb{B}_{\textbf{k},\varphi})) = \begin{cases}
			(x-1)^{\frac{\varphi}{2}-2}  (x+1)^{\textbf{k}-\frac{\varphi}{2}-1}  f(x), \textup{ if } \varphi \textup{ is even,} \\
			\\
			(x-1)^{\frac{\varphi-1}{2}-2}  (x+1)^{\textbf{k}-\frac{\varphi-1}{2}-1}  g(x), \textup{ if } \varphi \textup{ is odd,}
		\end{cases}
	\end{equation*}
 where $f(x)$ and $g(x)$ are polynomial as defined in Eqn.~\eqref{eqn:f_g}.
\end{theorem}
\begin{proof}
Let $\varphi$ be even, and $\varphi=2s$ for some $s$. 	Let us partition the vertex set of $\mathbb{B}_{\textbf{k},\varphi}$ as follows
	$$V(\mathbb{B}_{\textbf{k},\varphi}) = \{u_1,u_2,\cdots,u_{\textbf{k}-2s+1}\} \cup \{v\} \cup \{w_1^{(1)},w_2^{(1)}\} \cup \cdots \cup \{w_1^{(s-1)},w_2^{(s-1)}\},$$ where $v$ is the central cut vertex, $\{u_1,u_2,\cdots,u_{\textbf{k}-2s+1}\} \cup \{v\}$ is the vertex set of $K_{\textbf{k}-\varphi+2}$ and $\{v\} \cup \{w_1^{(i)},w_2^{(i)}\}$ is the vertex set of the $i^{th}$ $K_3$ for $1 \le i \le s-1$. Let us consider the following sets of vectors
	$$\mathcal{E}_1 = \{\mathbf{e}(u_1) - \mathbf{e}(u_j) : 2 \le j \le \textbf{k}-2s+1\} \cup \{\mathbf{e}(w_1^{(1)}) - \mathbf{e}(w_2^{(1)})\} \cup \cdots \cup \{\mathbf{e}(w_1^{(s-1)}) - \mathbf{e}(w_2^{(s-1)})\}$$ and
	$$\mathcal{E}_2 = \{\mathbf{e}(w_1^{(1)}) + \mathbf{e}(w_2^{(1)}) - \mathbf{e}(w_1^{(j)}) - \mathbf{e}(w_2^{(j)}) : 2 \le j \le s-1\},$$
	where $\mathbf{e}(u)$ is the column vector with $1$ at $u^{th}$ position and $0$ otherwise. We now  show that each of the vectors in $\mathcal{E}_1$ and $\mathcal{E}_2$ are eigenvectors of the adjacency matrix $\mathbf{A}(\mathbb{B}_{\textbf{k},\varphi})$ as defined in Eqn.~\eqref{eqn:A(B_n,k)}, and the eigenvalues can be observed as follows 
	$$\mathbf{A}(\mathbb{B}_{\textbf{k},\varphi}) \textbf{x} = - \textbf{x} \text{ for all } \textbf{x} \in \mathcal{E}_1 \mbox{ and }
	 \mathbf{A}(\mathbb{B}_{\textbf{k},\varphi}) \textbf{x} = \textbf{x} \text{ for all } \textbf{x} \in \mathcal{E}_2.$$ 
	 Thus, $-1$ is an eigenvalue of $\mathbf{A}(\mathbb{B}_{\textbf{k},\varphi})$ with multiplicity $\textbf{k}-s-1 = \textbf{k}-\frac{\varphi}{2}-1$ and $1$ is an eigenvalue of $\mathbf{A}(\mathbb{B}_{\textbf{k},\varphi})$ with multiplicity $s-2 = \frac{\varphi}{2}-2$. Using Eqn.~\eqref{eqn:ch_f}, it can be verified that $1$ and $-1$ are not roots of the characteristic polynomial $f(x)$ of  $\mathbf{Q}(\mathbf{A}(\mathbb{B}_{\textbf{k},\varphi}))$. In view of the above arguments, Theorem~\ref{thm:quesient} and Eqn.~\eqref{eqn:ch_f}, the result hold true if $\varphi$ is even.
	 
Next, let $\varphi$ be odd and $\varphi=2s+1$ for some $s$. Let us partition the vertex set of $\mathbb{B}_{\textbf{k},\varphi}$ as follows
$$V(\mathbb{B}_{\textbf{k},\varphi}) = \{u_1,u_2,\cdots,u_{\textbf{k}-2s}\} \cup \{v\} \cup \{w_1^{(1)},w_2^{(1)}\} \cup \cdots \cup \{w_1^{(s-1)},w_2^{(s-1)}\}\cup \{w_1^{(s)}\},$$ where $v$ is the central cut vertex, $\{u_1,u_2,\cdots,u_{\textbf{k}-2s}\} \cup \{v\}$ is the vertex set of $K_{\textbf{k}-\varphi+2}$, $\{v\} \cup \{w_1^{(i)},w_2^{(i)}\}$ is the vertex set of the $i^{th}$ $K_3$ for $1 \le i \le s-1$ and $\{v\} \cup \{w_1^{(s)}\}$ is the vertex set of $K_2$. Let us consider the following sets of vectors
$$\mathcal{F}_1 = \{\mathbf{e}(u_1) - \mathbf{e}(u_j) : 2 \le j \le \textbf{k}-2s\} \cup \{\mathbf{e}(w_1^{(1)}) - \mathbf{e}(w_2^{(1)})\} \cup \cdots \cup \{\mathbf{e}(w_1^{(s-1)}) - \mathbf{e}(w_2^{(s-1)})\}$$ and
$$\mathcal{F}_2 = \{\mathbf{e}(w_1^{(1)}) + \mathbf{e}(w_2^{(1)}) - \mathbf{e}(w_1^{(j)}) - \mathbf{e}(w_2^{(j)}) : 2 \le j \le s-1\},$$
where $\mathbf{e}(u)$ is the column vector with $1$ at $u^{th}$ position and $0$ otherwise. Arguing analogous to the case $\varphi$ is even, the result follows from  Theorem~\ref{thm:quesient} and Eqn.~\eqref{eqn:ch_g}.
\end{proof}

Recall, a graph on $2t+1$ vertices consisting of triangles ($K_3$) with central cut vertex  is called a friendship graph and is denoted by $F_{2t+1}$.  We  now calculate the  spectral radius $\rho(F_{2t+1})$ and observe a simple consequence.
\begin{lem}
Let $F_{2s+1}$ be a friendship graph. Then $\rho(F_{2s+1})= \dfrac{1+\sqrt{1+8s}}{2}.$
\end{lem}
\begin{proof}
Let $F_{2s+1}$ be a friendship graph on $2s+1$ vertices. Then, the adjacency matrix of $F_{2s+1}$ can be represented in the following block form
	\begin{equation*}
		\mathbf{A}(F_{2s+1}) = \begin{bmatrix}
			0 & \mathds{1}_{2s}^T \\
			\mathds{1}_{2s} & \mathbf{C}
		\end{bmatrix},
	\end{equation*}
	 where $\mathbf{C}$ is a block diagonal matrix of order $2s$ with $J_2-I_2$ as the $2 \times 2$ diagonal blocks. Then, the equitable quotient matrix of $\mathbf{A}(\mathbb{B}_{\textbf{k},\varphi})$ is given by $\ds 	
		\mathbf{Q}(\mathbf{A}(F_{2s+1})) =  \begin{bmatrix}
			0 & 2s \\
			1 & 1
		\end{bmatrix}$ 
and hence $p(x)=x^2-x-2s$ is the characteristic polynomial of	$\mathbf{Q}(\mathbf{A}(F_{2s+1}))$. Therefore,  the result follows from  Theorem~\ref{thm:quesient}.
\end{proof}
\begin{cor}\label{cor:friend}
If $\varphi$ is even, then the dissociation number of $F_{\varphi+1}$ is $\varphi$ and  $\rho(F_{\varphi+1})= \dfrac{1+\sqrt{1+4\varphi}}{2}.$
\end{cor}

%%%In the next section, we prove our main result that the maximum spectral radius $\rho(G)$  amongst all graphs $G$ in $\mathbf{Bl}(\textbf{k}, \varphi)$  is uniquely attained by $\mathbb{B}_{\textbf{k},\varphi}$.
%%%
%%%%\section{Main Result}\label{sec:main}
%%%In this section, we establish that a graph with maximal spectral radius in $\mathbf{Bl}(\textbf{k}, \varphi)$ has a central cut vertex and subsequently prove our main result.

We now prove two lemmas involving adjacency matrix of a graph, the spectral radius and a corresponding Perron vector. These lemmas will play a important role to show the existence and uniqueness of a graph with maximal spectral radius in $\mathbf{Bl}(\textbf{k}, \varphi)$.  Before proceeding further we recall the definition of graph isomorphism,  a related result and its consequences.

Let $G_1=(V(G_1), E(G_1))$ and $G_2=(V(G_2), E(G_2))$ be two graphs and $\pi: V(G_1)\rightarrow V(G_2) $ be a map. The map $\pi $ is said to be a graph homomorphism between $G_1$ and $G_2$, if  $ u \sim v$ in $G_1$ implies that $\pi(u) \sim \pi(v)$ in $G_2$. The map $\pi$ is said to be a graph isomorphism between $G_1$ and $G_2$, if $\pi$ is a graph homomorphism and bijective. Moreover, if $\pi$ is a graph isomorphism between $G_1$ and $G_2$, we say that two graphs $G_1$ and $G_2$ are isomorphic and denoted by $G_1 \cong G_2.$    Furthermore, if $G_1=G_2$ and $\pi$ is a graph isomorphism  between $G_1$ and $G_2$, then $\pi$ is called a graph automorphism of $G_1$. 

It is easy to see that if $G$ is a finite graph and $\pi$ is a graph automorphism of $G$, then $\pi$ is a permutation map on the set of vertices V(G) and we have the following result.

\begin{theorem}\label{prop:permu}\cite{Biggs}
Let $\mathbf{A}$ be the adjacency matrix of a  graph $G=(V(G),E(G))$  and $\pi$ is a permutation map of $V(G)$. Then, $\pi$ is a graph automorphism of $G$ if and only if $\mathbf{P}_\pi \mathbf{A}=\mathbf{A}\mathbf{P}_\pi$, where  $\mathbf{P}_\pi$ is the permutation matrix representing $\pi.$
\end{theorem}
The result below is an immediate application of the above theorem and hence presented as a corollary without proof. 
\begin{cor}\label{cor:auto}
Let $\mathbf{A}$ be the adjacency matrix of a graph $G=(V(G),E(G))$ and $(\rho(G),\mathbf{x})$ be an eigen-pair corresponding to the spectral radius $\rho(G)$. For $u,v \in V(G)$ and $u\neq v$, if $\pi$ is a graph automorphism of $G$ such that $\pi(u)=v$,  $\pi(v)=u$ and  $\pi(w)=w$ otherwise, then $x_u=x_v.$
\end{cor}
\begin{rem}
A stronger inference from Theorem~\ref{prop:permu} can be made in comparison to the Corollary~\ref{cor:auto}, and we state it as follows: Let $\mathbf{A}$ be the adjacency matrix of a graph $G=(V(G),E(G))$ and $(\rho(G),\mathbf{x})$ be an eigen-pair corresponding to the spectral radius $\rho(G)$. If $\pi$ is a graph automorphism of  $G$,  then $\mathbf{A}_\pi=\mathbf{P}_\pi \mathbf{A} \mathbf{P}_\pi^T $ is the adjacency matrix of $G$ under the vertex indexing due to $\pi$, and $\mathbf{x}_\pi = (x_{\pi(v)})_{v\in V(G)}$  is a Perron vector of $\rho(G)$.
\end{rem}

\begin{lem}\label{lem:comp}
  Let $\mathbf{A}$ be the adjacency matrix of a graph $G$ and $(\rho(G),\mathbf{x})$ be an eigen-pair corresponding to the spectral radius $\rho(G)$. For $n\geq 3$, let $K_n$ be a pendant block of $G$ on vertices $w_1,w_2,\ldots, w_n$  with $w_1$ as the cut vertex. Then, $x_{w_2}= x_{w_2}=\cdots=x_{w_n}$ and $x_{w_1} > x_{w_i}$ for $2\leq i \leq n.$
\end{lem}
\begin{proof}
 For $n\geq 3$, let $K_n$ be a pendant block of $G$ on vertices $w_1,w_2,\ldots, w_n$  with $w_1$ as the cut vertex. Then, using $\mathbf{A}\mathbf{x}=\rho(G)\mathbf{x}$, we have \begin{equation}\label{eqn:3.2}
\rho(G)x_{w_i}=x_{w_1}+ \sum_{j=2\atop j\neq i}^n x_{w_j} \mbox{ for  }  2\leq i \leq n.
\end{equation}
 Using $w_2,w_3,\ldots,w_n$  are not cut vertices and $K_n$ is a complete graph, we have the following: For every $2\leq i,j\leq n$ and $i\neq j$, there exists a graph automorphism $\pi:V(G)\rightarrow V(G)$ such that $\pi(w_i)=w_j$, $\pi(w_j)=w_i$ and $\pi(w)=w$ otherwise. Therefore, by Corollary~\ref{cor:auto}, we have 
\begin{equation}\label{eqn:3_3}
x_{w_2}=x_{w_3}=\cdots=x_{w_n}.
\end{equation}
Using Eqn.~\eqref{eqn:3_3}, the Eqn.~\eqref{eqn:3.2}  residues to 
$\ds \rho(G)x_{w_i}=x_{w_1}+ (n-2) x_{w_i} \mbox{ for  }  2\leq i \leq n. $ Thus,
\begin{equation}\label{eqn:3_4}
x_{w_1}= (\rho(G)- (n-2))  x_{w_i} \mbox{ for  }  2\leq i \leq n.
\end{equation}

It is known that $\rho(K_n)=n-1$. Since $K_n$ is block of $G$, by Lemma~\ref{lem:sr_edge}, we have $\rho(G)>\rho(K_n)=n-1$. Therefore, using $\rho(G)>n-1$ in  Eqn.~\eqref{eqn:3_4}, the desired result follows.
\end{proof}

\begin{lem}\label{lem:cut-sift}
Let $G$ be a graph and $H$ be a block in $G$. Let $u$ and $v$ be two cut vertices of $G$ that are in $H$.  Let $\mathbf{A}$ be the adjacency matrix of $G$ and $(\rho(G),\mathbf{x})$ be an eigen-pair corresponding to the spectral radius $\rho(G)$ such that $x_u \ge x_v$. If $G^*$ is a graph obtained from $G$ by removing all the blocks at $v$ and attaching these blocks at $u$, then  $\rho(G) \leq \rho(G^*)$. Moreover,  $\rho(G) < \rho(G^*)$, if a  Perron vector of $G^*$ is not a Perron vector of $G$.
\end{lem}

\begin{proof}
Let  $N_{G}^{\mathbf{Bl}}(u)= \{H, B_1,B_2, \cdots, B_s\} \mbox{ and } N_{G}^{\mathbf{Bl}}(v)= \{H, C_1,C_2, \cdots, C_t\}. $ For graph $G$, let $U$ be the set of vertices in $C_l$ for $1\leq l\leq t$  excluding the cut vertex $v$. From the hypothesis, for graph $G^*$, we have
 $$N_{G^*}^{\mathbf{Bl}}(u)= \{H, B_1,B_2, \cdots, B_s, C_1,C_2, \cdots, C_t\},$$ 
and $U$ is the set of vertices in $C_l$ for $1\leq l\leq t$ excluding the cut vertex $u$.   Let $\mathbf{A}^*$ be the adjacency matrix of $G^*$.  Then
	\begin{equation*}
		\frac{1}{2} \mathbf{x}^T (\mathbf{A}^*-\mathbf{A}) \mathbf{x}  = - \sum_{i\in U}x_v  x_{i}
 + \sum_{i\in U} x_u  x_{i} = \sum_{i\in U}(x_u -x_v) x_{i} \ge 0.
	\end{equation*} 
Hence,
\begin{equation}\label{eqn:3.1}
\rho(G)= \mathbf{x}^T \mathbf{A} \mathbf{x} \leq \mathbf{x}^T \mathbf{A}^* \mathbf{x} \leq \rho(G^*),
\end{equation}	
which implies that $\rho(G) \le \rho(G^*)$.	 

Further, let us assume that a Perron vector of $G^*$ is not a Perron vector of $G$.   Using $\rho(G)$ is simple, we get  $\mathbf{x}$ is not a Perron vector of $\rho(G^*)$. Therefore, $\mathbf{x}^T \mathbf{A}^* \mathbf{x} < \rho(G^*)$ and hence  using Eqn.~\eqref{eqn:3.1}, we get $\rho(G) < \rho(G^*)$.  This completes the  proof.
\end{proof}

\section{Maximization of Spectral Radius    in $\mathbf{Bl}(\textbf{k}, \varphi)$}\label{sec:maxi}

In this section, we build up the necessary tools in the form of a few results that help us to prove the maximum spectral radius $\rho(G)$  amongst all graphs $G$ in $\mathbf{Bl}(\textbf{k}, \varphi)$  is uniquely attained by $\mathbb{B}_{\textbf{k},\varphi}$. We begin with the following definition.

A graph $\widehat{G}$ in $ \mathbf{Bl}(\textbf{k}, \varphi)$  is said to be a graph with maximal spectral radius in $\mathbf{Bl}(\textbf{k}, \varphi)$ if $$\rho(G) \leq \rho(\widehat{G}) \mbox{ for all } G \in \mathbf{Bl}(\textbf{k}, \varphi).$$

\subsection{ Properties of a Graph with Maximal Spectral Radius in $\mathbf{Bl}(\textbf{k}, \varphi)$ }

In this subsection, we prove a few results that give us a few properties of a  graph with maximal spectral radius in $\mathbf{Bl}(\textbf{k}, \varphi)$. We begin with the following proposition.
%%
%%We now prove a few results that give us a few properties of a  graph with maximal spectral radius in $\mathbf{Bl}(\textbf{k}, \varphi)$,  and are useful for our subsequent results.

\begin{prop}\label{prop:1}
Let  $G \in \mathbf{Bl}(\textbf{k}, \varphi)$ and $\mathcal{D}$ be a maximal dissociation set of $G$.  Then $G$ is not a graph with maximal spectral radius in $\mathbf{Bl}(\textbf{k}, \varphi)$, if either of the following conditions hold true

\begin{itemize}
\item[$(i)$] $G$ has a block $B$ such that $|B \cap \mathcal{D}|=0$, 

\item[$(ii)$]  $\mathcal{D}$ contain a cut vertex of $G$,

\item [$(iii)$] $G$  have two adjacent blocks $B_1$ and $B_2$ such that $|B_1 \cap \mathcal{D}|=|B_2 \cap \mathcal{D}|=1$.
\end{itemize}
\end{prop}
\begin{proof}
Let $B$ be a block of $G$ such that $B\cap \mathcal{D}= \emptyset$, and $C$ be a block that is adjacent to $B$. Let $G^*$ be a block graph obtained from $G$ by adding edges such that the induced subgraph on the set of vertices on  $B\cup C$ is a complete graph. Thus, by  Lemma~\ref{lem:sr_edge}, we get $\rho(G)< \rho(G^*)$. Moreover, $B\cap \mathcal{D}= \emptyset$ gives  $G^* \in \mathbf{Bl}(\textbf{k}, \varphi)$. This proves part $(i)$.

Next, let $u$ be a cut vertex in $G$ such that $u \in \mathcal{D}.$ Let $N_{G}^{\mathbf{Bl}}(u)= \{ B_1,B_2, \cdots, B_s\}.$ 
Then, either $(B_i\setminus\{u\}) \cap \mathcal{D} = \emptyset$ for all $i=1,2,\ldots, s$ or there exists a unique $j \in \{1,2,\ldots, s\}$ such that $(B_j\setminus\{u\}) \cap \mathcal{D} \neq \emptyset.$ Let $G^{**}$ be a block graph obtained from $G$ by adding edges such that the induced subgraph on the set of vertices on  $\cup_{i=1}^s B_i$ is a complete graph. By construction, $G^{**} \in \mathbf{Bl}(\textbf{k}, \varphi)$, and using  Lemma~\ref{lem:sr_edge}, we get $\rho(G)< \rho(G^{**})$.  This proves part $(ii)$.

Finally to prove part $(iii)$, suppose on the contrary, let  $G$ be a graph with maximal spectral radius in $\mathbf{Bl}(\textbf{k}, \varphi)$ such that $G$  have two adjacent blocks $B_1$ and $B_2$ such that $|B_1 \cap \mathcal{D}|=|B_2 \cap \mathcal{D}|=1$. Let $u \in B_1 \cap \mathcal{D}$ and $v\in B_2 \cap \mathcal{D}$. Then, Part $(ii)$ yields that $u$ and $v$ are not cut vertices in $G$.

 Let $G^{***}$ be a block graph obtained from $G$ by adding edges such that the induced subgraph on the set of vertices on  $B_1\cup B_2$ is a complete graph. Then, by Lemma~\ref{lem:sr_edge}, we get $\rho(G)< \rho(G^{***})$. By construction, $u$ and $v$ are not cut vertices in $G^{***}$. Thus,  $\mathcal{D}$ is a maximal dissociation set of $G^{***}$ as well. Hence $G^{***} \in \mathbf{Bl}(\textbf{k}, \varphi)$,  which is a contradiction. This completes the proof.
\end{proof}

Before establishing a few more properties of a graph with maximal spectral radius in $\mathbf{Bl}(\textbf{k}, \varphi)$, we first prove two important lemmas. 

\begin{lem}\label{lem:Kn-Kn_1}
Let $G\in \mathbf{Bl}(\textbf{k}, \varphi)$ and $\mathcal{D}$ be a maximal dissociation set of $G$. Let $\mathbf{A}$ be the adjacency matrix of $G$ and $(\rho(G),\mathbf{x})$ be an eigen-pair corresponding to the spectral radius $\rho(G)$. For $m,n \ge 4$, let $K_m \circledcirc K_n$ be an induced subgraph of $G$ joined via the cut vertex $w$ such that $w \notin \mathcal{D}$, $K_m \cap \mathcal{D} =\{p,q \}$ and $K_n \cap \mathcal{D}=\{r,s  \}$. 
\begin{enumerate}
\item[(a)] Let $x_p \ge x_r\geq x_s$ and  $x_q \ge x_r \geq x_s$. If  $G^*$ is a block graph obtained from $G$ using the following operations (see Figure~\ref{fig:1.m,n>4}):
\begin{enumerate}
\item[$1.$] Delete  edges $rj$ and $sj$ for all  $j \in K_n \setminus \{r,s,w\}$,
\item[$2.$]  Add edge $ij$ for all $i \in K_m \setminus \{w\}$ and for all $j \in K_n \setminus \{r,s,w\}$,
\end{enumerate}
then  $G^*\in \mathbf{Bl}(\textbf{k}, \varphi)$ and $\rho(G) < \rho (G^*)$.
\item[(b)] Let $x_p \ge x_r\geq x_s$ and  $x_r>x_q$. If $G^{**}$ is a block graph obtained from $G$ using the following operations (see Figure~\ref{fig:1.m,n>4}):
\begin{enumerate}
\item[$1.$] Delete  edge $qi$   for all $i \in K_n \setminus \{q,w\}$, 
\item [$2.$]Delete edge $sj$ for all  $j \in K_n \setminus \{s,w\}$, 
\item[$3.$]  Add edges $qs$ and $ij$ for all $i \in K_m \setminus \{q,w\}$ and for all $j \in K_n \setminus \{s,w\}$,
\end{enumerate}
then,  $G^{**}\in \mathbf{Bl}(\textbf{k}, \varphi)$ and $\rho(G) < \rho (G^{**})$.
\end{enumerate}

\end{lem}
\begin{figure*}[ht]
	\centering
	\begin{tikzpicture}[scale=.6]
		\Vertex[x=1,y=-1,size=.1,label=$w$,position=left,color=red]{B}
		\Vertex[x=-0.6,y=-1,label=$K_{m}$,style={shape=coordinate}]{-1} 
		\Vertex[x=2.7,y=-1,label=$K_{n}$,style={shape=coordinate}]{-1}
		\Vertex[x=4.35,y=-1.5,size=.1,label=$s$,position=right,color=green]{D}
		\Vertex[x=4.35,y=-0.5,size=.1,label=$r$,position=right,color=green]{E}
		\Vertex[x=-2.35,y=-1.5,size=.1,label=$q$,position=left,color=green]{F}
		\Vertex[x=-2.35,y=-0.5,size=.1,label=$p$,position=left,color=green]{G}

		\draw[black, thick, dash dot, use Hobby shortcut,closed](C)(-2,0)..(-2,-2)..(1,-1);
		
		\draw[black, thick, dash dot, use Hobby shortcut,closed](C)(4,0)..(1,-1)..(4,-2);
		\node at (1,-4) {$G$};
	\end{tikzpicture}\hspace{0.4cm}
	\begin{tikzpicture}[scale=.6]
	\Vertex[x=1,y=-1,size=.1,label=$w$,position=above,color=red]{B}
	\Vertex[x=-0.7,y=-1,label=$K_{m+n-3}$,style={shape=coordinate}]{-1} 
	\Vertex[x=3.1,y=-1,label=$K_{3}$,style={shape=coordinate}]{-1}
	\Vertex[x=4.35,y=-2,size=.1,label=$s$,position=right,color=green]{D}
	\Vertex[x=4.35,y=0,size=.1,label=$r$,position=right,color=green]{E}
	\Vertex[x=-2.35,y=-1.5,size=.1,label=$q$,position=left,color=green]{F}
	\Vertex[x=-2.35,y=-0.5,size=.1,label=$p$,position=left,color=green]{G}
	
	\Edge[lw=2pt](B)(D)
	\Edge[lw=2pt](B)(E)
	\Edge[lw=2pt](E)(D)
	\node at (1,-4) {$G^*$};
	\draw[black, thick, dash dot, use Hobby shortcut,closed](C)(-2,0)..(-2,-2)..(1,-1);	
	\end{tikzpicture}\hspace{0.4cm}
	\begin{tikzpicture}[scale=.6]
	\Vertex[x=1,y=-1,size=.1,label=$w$,position=above,color=red]{B}
	\Vertex[x=-0.7,y=-1,label=$K_{m+n-3}$,style={shape=coordinate}]{-1} 
	\Vertex[x=3.1,y=-1,label=$K_{3}$,style={shape=coordinate}]{-1}
	\Vertex[x=4.35,y=-2,size=.1,label=$s$,position=right,color=green]{D}
	\Vertex[x=4.35,y=0,size=.1,label=$q$,position=right,color=green]{E}
	\Vertex[x=-2.35,y=-1.5,size=.1,label=$r$,position=left,color=green]{F}
	\Vertex[x=-2.35,y=-0.5,size=.1,label=$p$,position=left,color=green]{G}

	\Edge[lw=2pt](B)(D)
	\Edge[lw=2pt](B)(E)
	\Edge[lw=2pt](E)(D)
\node at (1,-4) {$G^{**}$};
\draw[black, thick, dash dot, use Hobby shortcut,closed](C)(-2,0)..(-2,-2)..(1,-1);
	\end{tikzpicture}
	\caption{ Graph operations in Lemma~\ref{lem:Kn-Kn_1}}\label{fig:1.m,n>4}
\end{figure*}
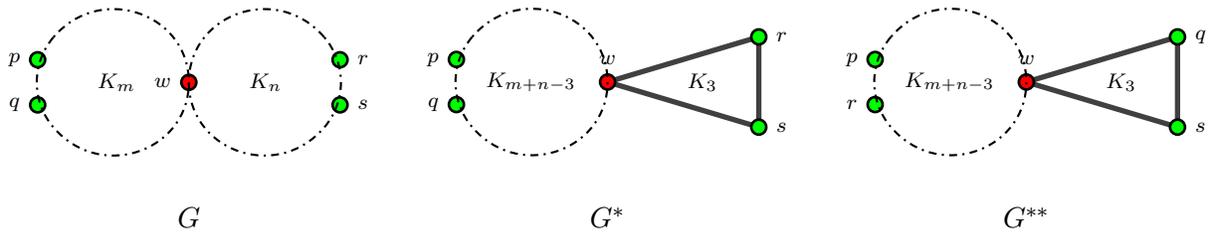
\begin{proof}
Using our hypothesis $m,n\geq 4$ gives us $K_m \setminus \{p,q,w\}$ and $K_n \setminus \{r,s,w\}$ are non empty sets. Let  $\mathbf{A}^*$ be the adjacency matrix of $G^*$.  Using   $x_p \ge x_r\geq x_s$ and  $x_q \ge x_r \geq x_s$, we have 
\begin{align*}
	\frac{1}{2} \mathbf{x}^T (\mathbf{A}^*-\mathbf{A}) \mathbf{x}  &= - \sum_{j \in K_n \setminus \{r,s,w\}} x_r x_j - \sum_{j \in K_n \setminus \{r,s,w\}} x_s x_j + \sum_{i \in K_m \setminus \{w\}} \sum_{j \in K_n \setminus \{r,s,w\}} x_i x_j\\
&> - \sum_{j \in K_n \setminus \{r,s,w\}} x_r x_j \, - \sum_{j \in K_n \setminus \{r,s,w\}}x_s x_j +  \sum_{j \in K_n \setminus \{r,s,w\}}x_p x_j +  \sum_{j \in K_n \setminus \{r,s,w\}} x_q x_j\\	
	&=   \sum_{j \in K_n \setminus \{r,s,w\}}\Big[(x_p - x_r) x_j  +  (x_q - x_s)x_j \Big]\\	
	& \ge 0.
\end{align*}
Thus, $\rho(G) < \rho (G^*)$. Further, let  $\mathbf{A}^{**}$ be the adjacency matrix of $G^{**}$. Using   $x_p \ge x_r\geq x_s$ and  $x_r>x_q$, we have 
\begin{align*}
	\frac{1}{2} \mathbf{x}^T (\mathbf{A}^{**}-\mathbf{A}) \mathbf{x}  &=x_sx_q  - \sum_{i \in K_m \setminus \{q,w\}}x_q x_i \, -   \sum_{j \in K_n \setminus \{s,w\}} x_s x_j +  \sum_{i \in K_m \setminus \{q,w\}}  \sum_{j \in K_n \setminus \{s,w\}} x_i x_j\\	
	&=x_sx_q  - \sum_{i \in K_m \setminus \{q,w\}}x_q x_i \, -   \sum_{j \in K_n \setminus \{s,w\}} x_s x_j \\
& \qquad \qquad	\qquad \qquad + \sum_{i \in K_m \setminus \{q,w\}} x_ix_r \, + \sum_{i \in K_m \setminus \{q,w\}}  \sum_{j \in K_n \setminus \{r, s,w\}} x_i x_j\\
& > x_sx_q  - \sum_{i \in K_m \setminus \{q,w\}}x_q x_i \, -   \sum_{j \in K_n \setminus \{s,w\}} x_s x_j \, + \sum_{i \in K_m \setminus \{q,w\}} x_ix_r \, +   \sum_{j \in K_n \setminus \{r, s,w\}} x_p x_j \\
& = x_sx_q  + \Big[ \sum_{i \in K_m \setminus \{q,w\}}(x_r-x_q) x_i \Big] \, + \Big[  \sum_{j \in K_n \setminus \{r,s,w\}}(x_p- x_s) x_j \Big] \, - x_sx_r \\
& =   \Big[ \sum_{i \in K_m \setminus \{q,w\}}(x_r-x_q) x_i \Big] \, + \Big[  \sum_{j \in K_n \setminus \{r,s,w\}}(x_p- x_s) x_j \Big] - x_s (x_r-x_q) \\
& =    \Big[ \sum_{i \in K_m \setminus \{p,q,w\}}(x_r-x_q) x_i \Big] \, + \,  x_p (x_r-x_q)\\
& \qquad \qquad	\qquad \qquad + \Big[  \sum_{j \in K_n \setminus \{r,s,w\}}(x_p- x_s) x_j \Big] - x_s (x_r-x_q) \\
& = (x_p-x_s) (x_r-x_q)\,  +  \Big[ \sum_{i \in K_m \setminus \{p,q,w\}}(x_r-x_q) x_i \Big] \, + \Big[  \sum_{j \in K_n \setminus \{r,s,w\}}(x_p- x_s) x_j \Big] \\
&>0.
\end{align*}

Thus, $\rho(G) < \rho (G^{**})$. It is easy to see that  $\mathcal{D}$ is also a maximal dissociation set for both $G^*$ and $G^{**}$ as well. This completes the proof.
\end{proof}

\begin{lem}\label{lem:Km-Kn_2}

Let $G\in \mathbf{Bl}(\textbf{k}, \varphi)$ and $\mathcal{D}$ be a maximal dissociation set of $G$. Let $\mathbf{A}$ be the adjacency matrix of $G$ and $(\rho(G),\mathbf{x})$ be an eigen-pair corresponding to the spectral radius $\rho(G)$. For $m,n \ge 4$, let $K_m \circledcirc K_n$ be an induced subgraph of $G$ joined via the cut vertex $w$ such that  $w \notin \mathcal{D}$, $K_m \cap \mathcal{D} =\{p,q \}$ and $K_n \cap \mathcal{D}=\{r \}$.

\begin{enumerate}
\item[(a)] Let $x_p\ge x_r$.  If $G^*$ is a block graph obtained from $G$ using the following operations (see Figure~\ref{fig:2.m,n>4}):
\begin{enumerate}
\item[$1.$] Delete  edge $rj$  for all  $j \in K_n \setminus \{r,w\}$,
\item[$2.$]  Add edge $ij$ for all $i \in K_m \setminus \{w\}$ and for all $j \in K_n \setminus \{r,w\}$,
\end{enumerate}
then  $G^*\in \mathbf{Bl}(\textbf{k}, \varphi)$ and $\rho(G) < \rho (G^*)$.

\item[(b)] Let $x_r > x_p,x_q$. If $G^{**}$ is a block graph obtained from $G$ using the following operations (see Figure~\ref{fig:2.m,n>4}):
\begin{enumerate}
\item[$1.$] Delete  edges $pi$ for all  $i \in K_m \setminus \{p,w\}$,
\item[$2.$]  Add edge  $ij$ for all $i \in K_m \setminus \{p,w\}$ and for all $j \in K_n \setminus \{w\}$,
\end{enumerate}
then  $G^{**}\in \mathbf{Bl}(\textbf{k}, \varphi)$ and $\rho(G) < \rho (G^{**})$.
\end{enumerate}
\end{lem}
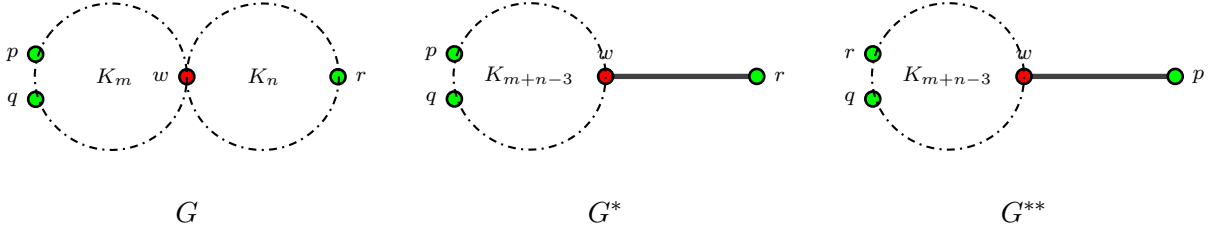
\begin{figure*}[ht]
	\centering
	\begin{tikzpicture}[scale=.6]
		\Vertex[x=1,y=-1,size=.1,label=$w$,position=left,color=red]{B}
		\Vertex[x=-0.6,y=-1,label=$K_{m}$,style={shape=coordinate}]{-1} 
		\Vertex[x=2.7,y=-1,label=$K_{n}$,style={shape=coordinate}]{-1}
		\Vertex[x=4.35,y=-1.0,size=.1,label=$r$,position=right,color=green]{E}
		\Vertex[x=-2.35,y=-1.5,size=.1,label=$q$,position=left,color=green]{F}
		\Vertex[x=-2.35,y=-0.5,size=.1,label=$p$,position=left,color=green]{G}

		\draw[black, thick, dash dot, use Hobby shortcut,closed](C)(-2,0)..(-2,-2)..(1,-1);
		
		\draw[black, thick, dash dot, use Hobby shortcut,closed](C)(4,0)..(1,-1)..(4,-2);
		\node at (1,-4) {$G$};
	\end{tikzpicture}\hspace{0.4cm}
	\begin{tikzpicture}[scale=.6]
		\Vertex[x=1,y=-1,size=.1,label=$w$,position=above,color=red]{B}
		\Vertex[x=-0.7,y=-1,label=$K_{m+n-3}$,style={shape=coordinate}]{-1} 
		%\Vertex[x=3.1,y=-1,label=$K_{3}$,style={shape=coordinate}]{-1}
		\Vertex[x=4.35,y=-1.0,size=.1,label=$r$,position=right,color=green]{E}
		\Vertex[x=-2.35,y=-1.5,size=.1,label=$q$,position=left,color=green]{F}
		\Vertex[x=-2.35,y=-0.5,size=.1,label=$p$,position=left,color=green]{G}

		\Edge[lw=2pt](B)(E)

		\node at (1,-4) {$G^*$};
		\draw[black, thick, dash dot, use Hobby shortcut,closed](C)(-2,0)..(-2,-2)..(1,-1);	
	\end{tikzpicture}\hspace{0.4cm}
	\begin{tikzpicture}[scale=.6]
		\Vertex[x=1,y=-1,size=.1,label=$w$,position=above,color=red]{B}
		\Vertex[x=-0.7,y=-1,label=$K_{m+n-3}$,style={shape=coordinate}]{-1} 
		%\Vertex[x=3.1,y=-1,label=$K_{3}$,style={shape=coordinate}]{-1}
		\Vertex[x=4.35,y=-1.0,size=.1,label=$p$,position=right,color=green]{E}
		\Vertex[x=-2.35,y=-1.5,size=.1,label=$q$,position=left,color=green]{F}
		\Vertex[x=-2.35,y=-0.5,size=.1,label=$r$,position=left,color=green]{G}

		\Edge[lw=2pt](B)(E)

	\node at (1,-4) {$G^{**}$};
	\draw[black, thick, dash dot, use Hobby shortcut,closed](C)(-2,0)..(-2,-2)..(1,-1);	
	\end{tikzpicture}
	\caption{ Graph operations in Lemma~\ref{lem:Km-Kn_2}}\label{fig:2.m,n>4}
\end{figure*}

\begin{proof} Using our hypothesis $m,n\geq 4$ gives us $K_m \setminus \{p,q,w\}$ and $K_n \setminus \{r,w\}$ are non empty sets. Let  $\mathbf{A}^*$ be the adjacency matrix of $G^*$. Using  $x_p\ge x_r$, we have 
\begin{align*}
     	\frac{1}{2} \mathbf{x}^T (\mathbf{A}^*-\mathbf{A}) \mathbf{x}   &= -\sum_{j \in K_n \setminus \{w\}} x_r  x_j + \sum_{i \in K_m \setminus \{w\}}  \sum_{j \in K_n \setminus \{r,w\}}x_i x_j\\
     	& > -  \sum_{j \in K_n \setminus \{w\}}x_r x_j +  \sum_{j \in K_n \setminus \{w\}} x_p x_j \\
     	&=  \sum_{j \in K_n \setminus \{w\}} (x_p - x_r) x_j\\
     	&\ge 0,
\end{align*}
which implies that $\rho(G) < \rho (G^*)$. By construction, $\mathcal{D}$ also a maximal dissociation set for both $G^*$ as well, this proves part $(a)$. The argument for the proof of part $(b)$ is similar to part $(a)$ and hence omitted.
\end{proof}

%%\begin{rem}
%%\end{rem}

\begin{prop}\label{prop:2}
Let $\widehat{G}$ be a graph with maximal spectral radius in $\mathbf{Bl}(\textbf{k}, \varphi)$. Then, the following hold true.

\begin{itemize}
\item[$(i)$] For $m,n \geq 4$, if $K_m$ and $K_n$ are two blocks of $\widehat{G}$, then $K_m$ and $K_n$ are not adjacent.
\item [$(ii)$] If $K_2$ or $K_3$ is a block of $\widehat{G}$, then it is a pendant block.
\end{itemize}
\end{prop}
\begin{proof}
Let $\widehat{G}$ be a graph with maximal spectral radius in $\mathbf{Bl}(\textbf{k}, \varphi)$ and $\mathcal{D}$ be a maximal dissociation set of $\widehat{G}$. Part $(i)$ follows from Lemmas~\ref{lem:Kn-Kn_1} and~\ref{lem:Km-Kn_2}. 

Next, let an edge $e$ $(=K_2)$ on vertices $u$ and $v$ be a block in $\widehat{G}$. If $e$ is not a pendant block, then both the vertices $u$ and $v$ are cut vertices. Since $\widehat{G}$ is a graph with maximal spectral radius in $\mathbf{Bl}(\textbf{k}, \varphi)$, by part $(ii)$ of  Proposition~\ref{prop:1}, we get $u,v \notin \mathcal{D}$ {\it i.e.,} $e \cap \mathcal{D}=\emptyset$. Hence part $(i)$ of Proposition~\ref{prop:1} leads to a contradiction. 

Finally, let a triangle $T$ $(=K_3)$ on vertices $u$, $v$ and $w$  be a block in $\widehat{G}$. If $T$ is not a pendant block, then by part $(i)$ and $(ii)$ of Proposition~\ref{prop:1}, exactly one of the vertices of $T$ is not a cut vertex in $\widehat{G}$. Without loss of generality, let us assume $w$ is not a cut vertex, $K_m$ and $K_n$ be two blocks of $\widehat{G}$ adjacent to $T$ with cut vertex $u$ and $v$, respectively. Since $\widehat{G}$ is a graph with maximal spectral radius in $\mathbf{Bl}(\textbf{k}, \varphi)$, using part $(i)$ of Proposition~\ref{prop:1}, we get $u,v\notin \mathcal{D}$. Hence   $w\in \mathcal{D}.$ Therefore, using  $|T\cap \mathcal{D}|=1$,  part $(i)$ and $(iii)$ of Proposition~\ref{prop:1} implies that $|K_m \cap \mathcal{D}|=|K_n\cap \mathcal{D}|=2$. Since $u,v \notin \mathcal{D}$, we get $m,n \geq 3.$

 Let $\widehat{\mathbf{A}}$ be the adjacency matrix of $\widehat{G}$ and $(\rho(\widehat{G}),\mathbf{x})$ be an eigen-pair corresponding to the spectral radius $\rho(\widehat{G})$.  Without loss of generality, let $x_u \geq x_v.$  We consider the following two cases to complete the proof.\\

\noindent {\bf \underline{Case $1$}:} $x_j<x_w$ for all $j \in K_n \setminus \{v\} $.\\

Since $|K_n\cap \mathcal{D}|=2,$ let $K_n\cap \mathcal{D}= \{p,q\}$.  Let  $G^{*}$ be a block graph obtained from $\widehat{G}$ using the following operations (see Figure~\ref{fig:case1_prop2}):
\begin{enumerate}
\item[$1.$] Delete  edge $pj$ for all $j \in K_n \setminus \{p\}$.
\item[$2.$] Add an edge $pu$.
\item[$3.$] Add edge $uj$ for all $j \in K_n \setminus \{v\}$ and $wj$ for all $j \in K_n \setminus \{p,v\}$.
\end{enumerate}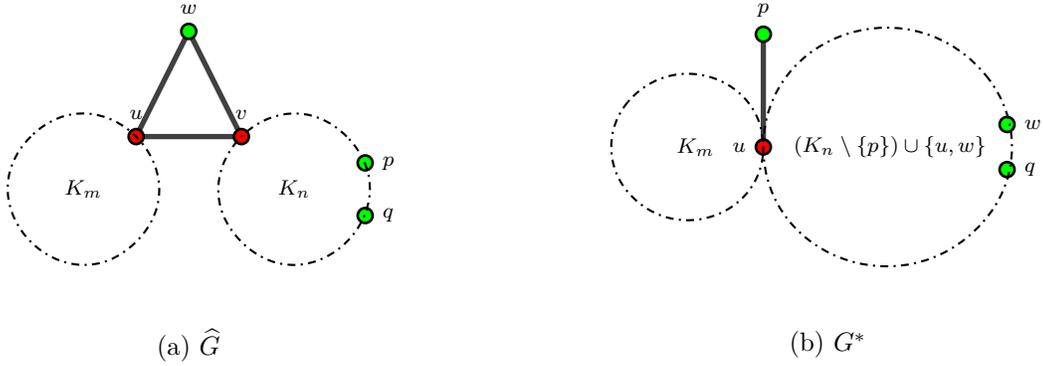
\begin{figure*}[ht]
\centering
\begin{subfigure}[t]{0.5\textwidth}
\centering
\begin{tikzpicture}[scale=.7]
\Vertex[size=.1,label=$u$,position=above,color=red]{A}
\Vertex[x=1,y=2,size=.1,label=$w$,position=above,color=green]{B}
\Vertex[x=2,y=0,size=.1,label=$v$,position=above,color=red]{C}
\Vertex[x=-1,y=-1,label=$K_{m}$,style={shape=coordinate}]{-1}
\Vertex[x=3,y=-1,label=$K_{n}$,style={shape=coordinate}]{-1}

\Vertex[x=4.35,y=-1.5,size=.1,label=$q$,position=right,color=green]{D}

\Vertex[x=4.35,y=-0.5,size=.1,label=$p$,position=right,color=green]{E}

\Edge[lw=2pt](A)(B)
\Edge[lw=2pt](A)(C)
\Edge[lw=2pt](B)(C)

\draw[black, thick, dash dot, use Hobby shortcut,closed](C)(-2,0)..(0,-2);

\draw[black, thick, dash dot, use Hobby shortcut,closed](C)(4,0)..(2,-2);

\end{tikzpicture}
\caption{$\widehat{G}$}
\end{subfigure}%
\begin{subfigure}[t]{0.5\textwidth}
\centering
\begin{tikzpicture}[scale=.6]
%\Vertex[size=.1,label=$u$,position=above,color=green]{A}
\Vertex[x=1,y=-1,size=.1,label=$u$,position=left,color=red]{B}
\Vertex[x=1,y=1.5,size=.1,label=$p$,position=above,color=green]{C}
\Vertex[x=-0.5,y=-1,label=$K_{m}$,style={shape=coordinate}]{-1}
%\Vertex[x=3.8,y=-1,label=$$,style={shape=coordinate}]{-1}

\Vertex[x=6.4,y=-1.5,size=.1,label=$q$,position=right,color=green]{D}
\Vertex[x=6.4,y=-0.5,size=.1,label=$w$,position=right,color=green]{E}

\Edge[lw=2pt](B)(C)

\node at (3.8,-1) {\scriptsize{$(K_{n}\setminus \{p\}) \cup \{u,w\}$}};

\draw[black, thick, dash dot, use Hobby shortcut,closed](C)(-2,0)..(-2,-2)..(1,-1);

\draw[black, thick, dash dot, use Hobby shortcut,closed](C)(6,0.5)..(1,-1)..(6,-2.5);
\end{tikzpicture}
\caption{$G^*$}
\end{subfigure}
\caption{Graph operation in Case $1$ of Proposition~\ref{prop:2}}\label{fig:case1_prop2}
\end{figure*}

By construction, $\mathcal{D}$ is also a maximal dissociation set of $G^*$ and hence $G^* \in \mathbf{Bl}(\textbf{k}, \varphi)$.   Let $\mathbf{A}^*$ be the adjacency matrix of $G^*$. Then,
\begin{align}\label{eqn:p2_1}
\frac{1}{2} \mathbf{x}^T (\mathbf{A}^*-\widehat{\mathbf{A}}) \mathbf{x} &= -x_px_v -\sum_{j \in K_n \setminus \{p,v\}} x_p x_j  + x_px_u \, + \sum_{j \in K_n \setminus \{v\}} x_u x_j \, +  \sum_{j \in K_n \setminus \{p,v\}}x_w x_j \nonumber \\ 
&>  -x_px_v - \sum_{j \in K_n \setminus \{p,v\}}x_p x_j \,  + x_px_u \, +  \sum_{j \in K_n \setminus \{p,v\}}x_w x_j \nonumber\\
&= x_p(x_u-x_v) + \sum_{j \in K_n \setminus \{p,v\}} (x_w-x_p) x_j. 
\end{align}
Since $x_j<x_w$ for all $j \in K_n \setminus \{v\} $, in particular $x_p < x_w$. Thus, using assumption $x_u \geq x_v$, Eqn.~\eqref{eqn:p2_1} yields that $\frac{1}{2} \mathbf{x}^T (\mathbf{A}^*-\widehat{\mathbf{A}}) \mathbf{x}>0.$ Hence $\rho(\widehat{G}) <  \rho(G^*)$. This leads to a contradiction to the maximality of $\widehat{G}$.\\

\noindent {\bf \underline{Case $2$}:} $x_j\geq x_w$ for some $j \in K_n \setminus \{v\} $.\\

 Using assumption, $x_j\geq x_w$ for some $j \in K_n \setminus \{v\} $, let $j_0 \in K_n \setminus \{v\}$ such that $x_{j_0}\geq x_w$. Moreover,  $n\geq 3$ gives  $|K_n \setminus \{v\}|\geq 2$, and hence $K_n \setminus \{j_0,v\}$ is non empty.\\

\begin{figure*}[ht]
	\centering
	\begin{subfigure}[t]{0.5\textwidth}
		\centering
		\begin{tikzpicture}[scale=.7]
			\Vertex[size=.1,label=$u$,position=above,color=red]{A}
			\Vertex[x=1,y=2,size=.1,label=$w$,position=above,color=green]{B}
			\Vertex[x=2,y=0,size=.1,label=$v$,position=above,color=red]{C} 
			\Vertex[x=-1,y=-1,label=$K_{m}$,style={shape=coordinate}]{-1} 
			\Vertex[x=3,y=-1,label=$K_{n}$,style={shape=coordinate}]{-1}
			
			\Edge[lw=2pt](A)(B)
			\Edge[lw=2pt](A)(C)
			\Edge[lw=2pt](B)(C)
			
			\draw[black, thick, dash dot, use Hobby shortcut,closed](C)(-2,0)..(0,-2);
			
			\draw[black, thick, dash dot, use Hobby shortcut,closed](C)(4,0)..(2,-2);
			
		\end{tikzpicture}
		\caption{$\widehat{G}$}
	\end{subfigure}%
	\begin{subfigure}[t]{0.5\textwidth}
		\centering
		\begin{tikzpicture}[scale=.6]
			%\Vertex[size=.1,label=$u$,position=above,color=green]{A}
			\Vertex[x=1,y=-1,size=.1,label=$u$,position=left,color=red]{B}
			\Vertex[x=1,y=1.5,size=.1,label=$w$,position=above,color=green]{C} 
			\Vertex[x=-0.5,y=-1,label=$K_{m}$,style={shape=coordinate}]{-1} 
			\Vertex[x=3.1,y=-1,label=$K_{n} \cup \{u\}$,style={shape=coordinate}]{-1}
			
			\Edge[lw=2pt](B)(C)
			
			\draw[black, thick, dash dot, use Hobby shortcut,closed](C)(-2,0)..(-2,-2)..(1,-1);
			
			\draw[black, thick, dash dot, use Hobby shortcut,closed](C)(5,0)..(1,-1)..(5,-2);
		\end{tikzpicture}
		\caption{$G^{**}$}
	\end{subfigure}
	\caption{Graph operation in Case~$2$ of Proposition~\ref{prop:2}}\label{fig:case2_prop2}
\end{figure*}
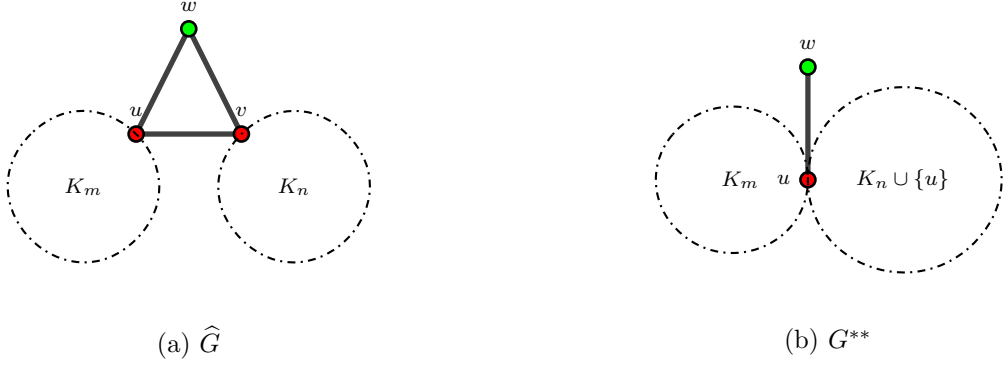 

\noindent Let  $G^{**}$ be a block graph obtained from $\widehat{G}$  using the following operations (see Figure~\ref{fig:case2_prop2}):
\begin{enumerate}
\item[$1.$] Delete the edge $vw$.
\item[$2.$]  Add edge $uj$ for all $j \in K_n \setminus \{v\}$.
\end{enumerate}
By construction, $G^{**} \in \mathbf{Bl}(\textbf{k}, \varphi)$. Let $\mathbf{A}^{**}$ be the adjacency matrix of $G^{**}$. Then,
\begin{align}\label{eqn:p2_2}
\frac{1}{2} \mathbf{x}^T (\mathbf{A}^{**}-\widehat{\mathbf{A}}) \mathbf{x} &= -x_w x_v + \sum_{j \in K_n \setminus \{v\}} x_u x_j \nonumber \\
&=  -x_wx_v + x_u x_{j_0}+ \sum_{j \in K_n \setminus \{j_0, v\}}x_u x_j.
\end{align}
 Using $x_u \geq x_v$ and $x_{j_0}\geq x_w$, we get $x_u x_{j_0} \geq x_wx_v$. Therefore, using $K_n \setminus \{j_0,v\}$ is non empty in  Eqn.~\eqref{eqn:p2_2} gives $\frac{1}{2} \mathbf{x}^T (\mathbf{A}^{**}-\widehat{\mathbf{A}}) \mathbf{x}>0.$  Hence $\rho(\widehat{G}) <  \rho(G^{**})$, which is a contradiction.\\

From the above cases, it is clear that we reach a contradiction due to our assumption that triangle $T$ is not a pendant block. This completes the proof.
\end{proof}

\begin{cor}\label{cor:max_1}
If $\widehat{G}$ is a graph with maximal spectral radius in $\mathbf{Bl}(\textbf{k}, \varphi)$, then $\widehat{G}$ has a block $B$ such that every other block of $\widehat{G}$ is adjacent to $B$. 
\end{cor}
\begin{proof}
 Using $K_2$ and $K_3$ can only be a pendant block in $\widehat{G}$, if this result is not true, then there exist two adjacent blocks $ K_n$ and $K_m$, where $n,m\geq 4$. Which is a contradiction to part $(i)$ of Proposition~\ref{prop:2}. This completes the proof.
\end{proof}

Using the above results and properties of a graph with maximal spectral radius in $\mathbf{Bl}(\textbf{k}, \varphi)$, we now prove that a graph with maximal spectral radius in $\mathbf{Bl}(\textbf{k}, \varphi)$ has a central cut vertex.

\begin{prop}\label{prop:3}
If $\varphi>2$ and $\widehat{G}$ is a graph with maximal spectral radius in $\mathbf{Bl}(\textbf{k}, \varphi)$, then $\widehat{G}$ is block graph with a central cut vertex.
\end{prop}
\begin{proof}
Let $\widehat{G}$ be a graph with maximal spectral radius in $\mathbf{Bl}(\textbf{k}, \varphi)$ and $\mathcal{D}$ be a maximal dissociation set of $\widehat{G}$. Using Corollary~\ref{cor:max_1}, let $B$ be a block of $\widehat{G}$ such that every other block of $\widehat{G}$ is adjacent to $B$. By part $(ii)$  of Proposition~\ref{prop:2}, $K_2$ and $K_3$ can only be a pendant block in $\widehat{G}$. Thus, if $B$ is $K_2$ or $K_3$, then we are done. 

Next, let $B=K_n$, where $n\geq 4$. Suppose $\widehat{G}$ have more than one cut vertex. We now present two observations on $\widehat{G}$.
\begin{itemize}
\item[(a)] Using $\varphi >2$, $\widehat{G}$ has at least one more block other than $B$. Since  $B$ is adjacent to every other block of $\widehat{G}$, part $(i)$ of Proposition~\ref{prop:2} implies that the blocks other than $B$ can only be $K_2$ or $K_3.$ 

\item[(b)] By part $(ii)$ of Proposition~\ref{prop:1}, the cut vertices of $\widehat{G}$ are not in $\mathcal{D}$, and hence the elements of $B\cap \mathcal{D}$ are not cut vertices. Furthermore, if $K_2$ is a block, then $|K_2 \cap \mathcal{D}|=1$, and if $K_3$ is a block, then $|K_3\cap \mathcal{D}|=2$.
\end{itemize}
 By Proposition~\ref{prop:1}, $B \cap \mathcal{D} $ is non-empty,  and its elements are not cut vertices of $\widehat{G}$. Thus, we now consider the cases $|B \cap \mathcal{D}|=1$ and $|B \cap \mathcal{D}|=2$ to complete the proof.\\

\noindent {\bf \underline{Case $1$}:} $|B \cap \mathcal{D}|=1$\\

In this case,  if $K_m$ is a block adjacent to $B$, then $m=3.$  Otherwise, if $m=2$,  then $\widehat{G}$ has two adjacent blocks $B$ and $K_m$ such that $|B \cap \mathcal{D}|=1$ and $|K_2 \cap \mathcal{D}|=1$. Thus, by part $(iii)$ of Proposition~\ref{prop:1}, $\widehat{G}$ is not a graph with maximal spectral radius in $\mathbf{Bl}(\textbf{k},\varphi)$, which is a contradiction to our hypothesis. Therefore, all the blocks of $\widehat{G}$ other than $B$ are $K_3$.

Since $|B\cap \mathcal{D}|=1$, all the vertices of $B$ except one are cut vertices. For $n\geq 4$, let $B=K_n$ be on vertices $w_1,w_2,,\ldots,w_n.$ Without loss of generality, let $w_1,w_2,,\ldots,w_{n-1}$ be cut vertices and $w_n\in \mathcal{D}$ such that
\begin{equation}\label{eqn:cut_order}
x_{w_1}\leq x_{w_2}\leq \cdots \leq x_{w_{n-1}}.
\end{equation}

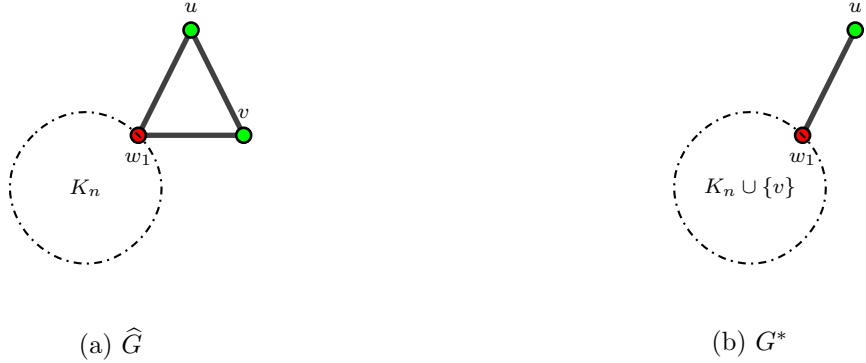
\begin{figure*}[ht]
	\centering
	\begin{subfigure}[t]{0.5\textwidth}
		\centering
		\begin{tikzpicture} [scale=.7]
			\Vertex[size=.1,label=$w_1$,position=below,color=red]{A}
			\Vertex[x=1,y=2,size=.1,label=$u$,position=above,color=green]{B}
			\Vertex[x=2,y=0,size=.1,label=$v$,position=above,color=green]{C} 
			\Vertex[x=-1,y=-1,label=$K_{n}$,style={shape=coordinate}]{-1}

			\Edge[lw=2pt](A)(B)
			\Edge[lw=2pt](A)(C)
			\Edge[lw=2pt](B)(C)
			
			\draw[black, thick, dash dot, use Hobby shortcut,closed](C)(-2,0)..(0,-2);
			
			%\draw[black, thick, dash dot, use Hobby shortcut,closed](C)(4,0)..(2,-2);
			
		\end{tikzpicture}
		\caption{$\widehat{G}$}
	\end{subfigure}%
	\begin{subfigure}[t]{0.5\textwidth}
		\centering
		\begin{tikzpicture} [scale=.7]
			\Vertex[size=.1,label=$w_1$,position=below,color=red]{A}
			%\Vertex[x=0.5,y=-1,size=.1,label=$w_1$,position=below,color=black]{B}
			\Vertex[x=1,y=2,size=.1,label=$u$,position=above,color=green]{C} 
			\Vertex[x=-1.0,y=-1,label=$K_{n} \cup \{v\}$,style={shape=coordinate}]{-1} 
			
			\Edge[lw=2pt](A)(C)
			%\Edge[lw=2pt](B)(C)
			
			\draw[black, thick, dash dot, use Hobby shortcut,closed](C)(-2,0)..(0,-2);
			
			%\draw[black, thick, dash dot, use Hobby shortcut,closed](C)(4,0)..(1,-1)..(4,-2);
		\end{tikzpicture}
		\caption{$G^*$}
	\end{subfigure}
	\caption{ Graph operation in Case~$1$ of Proposition~\ref{prop:3} } \label{fig:Case1_prop3}
\end{figure*}

Let $K_3$ be a pendant block on vertices $u,v$ and $w_1$, and it is adjacent to  $B$ via the cut vertex $w_1$. Using Lemma~\ref{lem:comp}, we have 
\begin{equation}\label{eqn:p3_0}
x_u=x_v \mbox{ and } x_{w_1}> x_u.
\end{equation}
Let  $G^{*}$ be a block graph obtained from $\widehat{G}$ using the following operations  (see Figure~\ref{fig:Case1_prop3}):
\begin{enumerate}
\item[$1.$] Delete the edge $uv$.
\item[$2.$] Add  edge $vj$ for all $j \in K_n \setminus \{w_1\}$. 
\end{enumerate}
By construction, $\mathcal{D}$ is also a maximal dissociation set of $G^*$ and hence $G^* \in \mathbf{Bl}(\textbf{k}, \varphi)$.   Let $\mathbf{A}^*$ be the adjacency matrix of $G^*$. 

\begin{align}\label{eqn:p3_1}
\frac{1}{2} \mathbf{x}^T (\mathbf{A}^*-\widehat{\mathbf{A}}) \mathbf{x} &= -x_ux_v + \sum_{j \in K_n \setminus \{w_1\}} x_v x_j  \nonumber \\ 
&= -x_ux_v + x_vx_{w_2}+\sum_{j \in K_n \setminus \{w_1,w_2\}} x_v x_j \nonumber\\
&= x_v(x_{w_2}-x_u) +\sum_{j \in K_n \setminus \{w_1,w_2\}} x_v x_j. 
\end{align}
Using   Eqns.~\eqref{eqn:cut_order} and~\eqref{eqn:p3_0},  we get   $x_{w_1}\leq x_{w_2}$ and $x_u < x_{w_1}$, respectively.  Hence $ x_{w_2}-x_u > 0$. Furthermore, $K_n \setminus \{w_1,w_2\}$ is  non-empty and hence the Eqn.~\eqref{eqn:p3_1} gives us $\frac{1}{2} \mathbf{x}^T (\mathbf{A}^*-\widehat{\mathbf{A}}) \mathbf{x} >0$. Therefore, $\rho(\widehat{G})<\rho( G^*)$. This is a contradiction  to our hypothesis that $\widehat{G}$ is a maximal graph.\\

\noindent {\bf \underline{Case $2$}:} $|B \cap \mathcal{D}|=2$\\

Depending on the number of cut vertices, by repeated application of  Lemma~\ref{lem:cut-sift} on $\widehat{G}$, we find a block graph $G^*$ with exactly one cut vertex such that  
\begin{equation}\label{eqn:p3_2}
\rho(\widehat{G}) \leq \rho(G^*).
\end{equation}
Since elements of  $B \cap \mathcal{D}$  are not cut vertices, and $|B\cap \mathcal{D}|=2$, under the graph operation used in Lemma~\ref{lem:cut-sift}, it is easy to see that $\mathcal{D}$ is a maximal dissociation set of $G^*$. Hence $G^* \in  \mathbf{Bl}(\textbf{k}, \varphi)$. 

Observe that, using graph operation in Lemma~\ref{lem:cut-sift}, we obtained $G^* \in \mathbf{Bl}(\textbf{k}, \varphi)$ with exactly one cut vertex from $\widehat{G} \in \mathbf{Bl}(\textbf{k}, \varphi)$ with two or more cut vertices. In this process, we reduce one cut vertex at a time from the given set of cut vertices in $\widehat{G}$. Moreover, we have applied these graph operations in such a way that the blocks in $\widehat{G}$ and $G^*$ are the same, but the two graphs only defer how these blocks are arranged with respect to the cut vertices. Thus,  $B$ is a block in both $\widehat{G}$ and $G^*$. 

For $n\geq 4$, let $B=K_n$ be on vertices $w_1,w_2,\ldots,w_n.$  Since  $G^*$ has exactly one cut vertex, without loss of generality, let $w_1$ be the cut vertex in $G^*$. Similarly, using the fact that  $\widehat{G}$ has two or more cut vertices, let $w_1$ and $w_2$ be two cut vertices of $\widehat{G}$. Moreover,   in $\widehat{G}$, $|B \cap \mathcal{D}|=2$ implies that at least two of the vertices in $B$ are not cut vertices, and without loss of generality, let $w_n$ be one such vertex. Then, $$N_{\widehat{G}}(w_n) =\{w_1,w_2,\ldots w_{n-1}\} \mbox{ and } N_{\widehat{G}}{(w_2)\setminus \{w_1,w_3,\ldots,w_n\}}\neq \emptyset.$$

Let $\mathbf{A}^*$ be the adjacency matrix of $G^*$ and $(\rho(G^*),\mathbf{y})$ be an eigen-pair corresponding to the spectral radius $\rho(G^*)$. Since $w_1$ is the only cut vertex in $G^*$,  $B=K_n$ is a pendant block of $G^*$ via the cut vertex $w_1$.  Therefore, by Lemma~\ref{lem:comp}, we have 
\begin{equation}\label{eqn:p3_3}
y_{w_2}=y_{w_3}=\cdots=y_{w_n}.
\end{equation}

Let $\widehat{\mathbf{A}}$ be the adjacency matrix of $\widehat{G}$. We claim that $(\rho(\widehat{G}),\mathbf{y})$ is not an eigen-pair corresponding to the spectral radius $\rho(\widehat{G})$. Suppose on the contrary, let us assume $\widehat{\mathbf{A}}\mathbf{y}=\rho(\widehat{G})\mathbf{y}.$ Thus, for every vertex $w$ in $\widehat{G}$, we have
$$(\widehat{\mathbf{A}}\mathbf{y})_w=\rho(\widehat{G})y_w.$$
In particular, for $w=w_n$, we have
\begin{equation}\label{eqn:p3_4}
 \rho(\widehat{G})y_{w_n}= \sum_{i=1\atop i\neq n}^n y_{w_i},
\end{equation}
and for $w=w_2$, we have
\begin{align}\label{eqn:p3_5}
 \rho(\widehat{G})y_{w_2}=& \sum_{v\sim w_2} y_v \nonumber\\ 
 =&\sum_{i=1\atop i\neq 2}^n y_{w_i} +  \sum_{v \in N_{\widehat{G}}{(w_2)\setminus \{w_1,w_3,\ldots,w_n\}}} y_v. 
\end{align}
From Eqn.~\eqref{eqn:p3_3}, we have $y_{w_2}=y_{w_n}$ and hence subtracting Eqn.~\eqref{eqn:p3_4} from Eqn.~\eqref{eqn:p3_5}, we get 
\begin{equation}\label{eqn:p3_6}
 \sum_{v \in N_{\widehat{G}}{(w_2)\setminus \{w_1,w_3,\ldots,w_n\}}} y_v = 0.
\end{equation}
Since $N_{\widehat{G}}{(w_2)\setminus \{w_1,w_3,\ldots,w_n\}}$ is non-empty and $\mathbf{y}$ is a Perron vector, we have
\begin{equation*}\label{eqn:p3_7}
 \sum_{v \in N_{\widehat{G}}{(w_2)\setminus \{w_1,w_3,\ldots,w_n\}}} y_v  > 0,
\end{equation*}
 which is a contradiction to Eqn.~\eqref{eqn:p3_6}. Thus, $\mathbf{y}$ is not an eigenvector  corresponding to $\rho(\widehat{G})$. Therefore, Lemma~\ref{lem:cut-sift} implies that the inequality in Eqn.~\eqref{eqn:p3_2} is strict, {\emph i.e.,} $\rho(\widehat{G}) < \rho(G^*).$ This leads to a  contradiction to our hypothesis that $\widehat{G}$ is a maximal graph.
 
The above cases leads to contradiction due to our assumption that $\widehat{G}$ has more than one cut vertex. This completes the proof.
\end{proof}

From the above result, all the blocks of a graph with maximal spectral radius in $\mathbf{Bl}(\textbf{k},\varphi)$ are adjacent to each other via a central cut vertex. Therefore, in view of part $(i)$ of Proposition~\ref{prop:2}, the maximal graph does not contain two blocks of size greater than equal to $4$, and is presented as a corollary.

\begin{cor}\label{cor:2}
Let $\widehat{G}$ be a graph with maximal spectral radius in $\mathbf{Bl}(\textbf{k},\varphi)$. Then, $\widehat{G}$ has at most one block $K_n$, where $n\geq 4.$
\end{cor}

\subsection{Existence and Uniqueness of  a Graph with Maximal Spectral Radius in $\mathbf{Bl}(\textbf{k}, \varphi)$ }

With the above necessary tools and properties of a graph with maximal spectral radius in $\mathbf{Bl}(\textbf{k}, \varphi)$, we are now ready to prove the existence and uniqueness of a graph with maximal spectral radius in $\mathbf{Bl}(\textbf{k}, \varphi)$.

\begin{theorem}\label{thm:mainthm}
A block graph $G$  is a graph with maximal spectral radius in $\mathbf{Bl}(\textbf{k}, \varphi)$ if and only if $G \cong \mathbb{B}_{\textbf{k},\varphi}.$ 
\end{theorem}

\begin{proof}
For  $\varphi=2$, if $\textbf{k}\geq 2$ and $\textbf{k}\neq 3$, then  $\mathbf{Bl}(\textbf{k}, \varphi)$  consist of only $K_{\textbf{k}}$, and if $\textbf{k}=3$, then $\mathbf{Bl}(\textbf{k}, \varphi)$ consist of $K_3$ or a path of length $2$. Therefore, if $\varphi=2$, then the maximal spectral radius in $\mathbf{Bl}(\textbf{k}, \varphi)$ uniquely attend by $K_{\textbf{k}} \cong \mathbb{B}_{\textbf{k},2}$. 

%%Next, let $\varphi >2$. In view of Proposition~\ref{prop:3}, we only need to show that if $\widehat{G}$ is a graph with maximal spectral radius in $\mathbf{Bl}(\textbf{k}, \varphi)$ and  $\widehat{G}$ has a central cut vertex, then $\widehat{G} \cong \mathbb{B}_{\textbf{k},\varphi}.$

For $\varphi >2$, let $\widehat{G}$ be a graph with maximal spectral radius in $\mathbf{Bl}(\textbf{k}, \varphi)$ and $\mathcal{D}$ be a maximal dissociation set of $\widehat{G}$. Using Proposition~\ref{prop:3},  $\widehat{G}$ has exactly one cut vertex. Moreover, by part $(ii)$ of Proposition~\ref{prop:1}, the unique cut vertex of $\widehat{G}$ is not in $\mathcal{D}$. Therefore, if $B$ is a block of $\widehat{G}$, then 
\begin{align}\label{eqn:size}
		|B\cap \mathcal{D}| =
		\begin{cases}
			1 & \text{ if } B=K_2,\\
			\\
			2  & \text{ if } B=K_n \mbox{ where } n\geq 3.
		\end{cases}
	\end{align}
We now consider the following two cases to complete the proof.\\

\noindent {\bf \underline{Case $1$}:} $\varphi$ is even.\\

In this case, we claim that $K_2$ is not a block of $\widehat{G}$.  Otherwise, using  Eqn.~\eqref{eqn:size} and $\varphi$ is even, there are even number $2t$ (say) of  $K_2$ blocks in $\widehat{G}$. Let $G^*$ be a block graph obtained from $\widehat{G}$ by adding additional edges on vertices $2t$ blocks  $K_2$  such that $G^*$ have $t$ blocks of $K_3$ instead of $2t$  blocks of  $K_2$. Thus, by construction $\mathcal{D}$ is also  a maximal dissociation set of $G^*$ and hence $G^*\in \mathbf{Bl}(\textbf{k}, \varphi)$. Further,  by Lemma~\ref{lem:sr_edge},  we have $\rho(\widehat{G}) < \rho(G^*)$. Which is a contradiction.

Let $\varphi=2s$  for some $s$. Using $K_2$ is not a block of $\widehat{G}$ and Eqn.~\eqref{eqn:size}, we have $\widehat{G}$ has exactly $s$ blocks. By Corollary~\ref{cor:2}, $\widehat{G}$ has at most one block  $K_n$, where $n\geq 4$ which implies that $\widehat{G}$ has at least $s-1$ blocks of $K_3$. Thus, the partition of $\textbf{k}$ vertices of $\widehat{G}$ into cut vertices and non-cut vertices from different blocks is given by $\textbf{k}= 1 + 2(s-1)+ (\textbf{k}-2s+1),$
\emph{i.e.,} $\widehat{G}$ has a central cut vertex, $s-1$ blocks of $K_3$ and one block of $K_{\textbf{k}-2s+2}$. Hence  $\widehat{G}\cong K_3^{(\frac{\varphi-2}{2})}\circledcirc K_{\textbf{k}-\varphi+2}. $\\

\noindent {\bf \underline{Case $2$}:} $\varphi$ is odd.\\

In this case, we claim that  $\widehat{G}$ has exactly one block as $K_2$.  Otherwise, using  Eqn.~\eqref{eqn:size} and $\varphi$ is odd, there are odd number $2t+1$ (say) of  $K_2$ blocks in $\widehat{G}$. Let $G^*$ be a block graph obtained from $\widehat{G}$ by adding additional edges on vertices $2t+1$ blocks  $K_2$  such that $G^*$ have $t$ blocks of $K_3$ and one block of $K_2$ instead of $2t+1$  blocks of  $K_2$. Thus,  arguments similar to Case~$1$ lead to a contradiction.

Let $\varphi=2s+1$  for some $s$. Since $\widehat{G}$ has exactly one block as $K_2$, using  Eqn.~\eqref{eqn:size}  we conclude that $\widehat{G}$ has exactly $s+1$ blocks. Using arguments similar to Case~$1$, the partition of $\textbf{k}$ vertices of $\widehat{G}$ into cut vertices and non cut vertices from different blocks is given by  $\textbf{k}= 1 + 1+ 2(s-1)+ (\textbf{k}-2s),$ \emph{i.e.,} $\widehat{G}$ has a central cut vertex, exactly one block as  $K_2$, $s-1$ blocks of $K_3$ and one block of $K_{\textbf{k}-2s+2}$. Hence  $\widehat{G}\cong K_2 \circledcirc K_3^{(\frac{\varphi-3}{2})}\circledcirc K_{\textbf{k}-\varphi+2}. $  \end{proof}

In the next section, we consider different cases to obtain possible bounds on the spectral radius $\rho(\mathbb{B}_{\textbf{k},\varphi})$ for the extremal graph $\mathbb{B}_{\textbf{k},\varphi}$ in  $\mathbf{Bl}(\textbf{k}, \varphi)$.

\section{Bounds on $\rho(\mathbb{B}_{\textbf{k},\varphi})$}\label{sec:bounds}

In this section, we provide bounds for the spectral radius $\rho(\mathbb{B}_{\textbf{k},\varphi})$ for the extremal graph $\mathbb{B}_{\textbf{k},\varphi}$. From the definition of $\mathbb{B}_{\textbf{k},\varphi}$, we know that $\varphi=2$ if and only if $\mathbb{B}_{\textbf{k},\varphi} \cong K_{\textbf{k}}$. By Lemma~\ref{lem:roots}, $\varphi=2$ if and only if $\rho(\mathbb{B}_{\textbf{k},\varphi})=\textbf{k}-1.$ Therefore, we provides bound $\rho(\mathbb{B}_{\textbf{k},\varphi})$ whenever $2<\varphi \leq \textbf{k}-1.$ 

We obtained the upper bound for  $\rho(\mathbb{B}_{\textbf{k},\varphi})$  using a simple observation and is presented below as a lemma without proof.

\begin{lem}\label{lem:inteq}
 if $\theta =\omega + \delta$ is a real root of a polynomial $h(x)$, then $\delta$ is root of $h(\omega +x)$. Furthermore, if $p(\delta)<0$, where $p(x)=ax^2+bx+c$ be real quadratic polynomial with $a>0$ and $c<0$, then 
 $$\delta < \dfrac{-b + \sqrt{b^2-4ac}}{2a},  \mbox{ and hence }  \theta < \omega +   \dfrac{-b + \sqrt{b^2-4ac}}{2a}.$$
\end{lem}

We begin with the following observation:  It is easy to see that $K_{\textbf{k}-\varphi +2}$ and $K_{1,\textbf{k}-1}$ are proper subgraphs of $\mathbb{B}_{\textbf{k},\varphi}$. Using  $\rho(K_{\textbf{k}-\varphi +2})= \textbf{k}-\varphi +1 $, $\rho(K_{1,\textbf{k}-1})= \sqrt{\textbf{k}-1} $ and   Lemma~\ref{lem:sr_edge}, we get
\begin{equation}\label{eqn:alpha_beta1}
\rho(\mathbb{B}_{\textbf{k},\varphi})\geq \max \big\{ \textbf{k}-\varphi +1, \sqrt{\textbf{k}-1}\big\}.
\end{equation}
Equality in Eqn.~\eqref{eqn:alpha_beta1} attains if and only if $\varphi=2.$

Let  $\alpha = \textbf{k}-\varphi+1$ and $\beta=\sqrt{\textbf{k}- 1} $. We states  a few immediate observations: $(1)$ $\alpha$ is an integer and $2\leq \alpha \leq \textbf{k}-1,$ $(2)$ $\varphi >2$ implies that $\textbf{k} >2$ and hence $\beta^2 > \beta.$ Before proving our main result of this section, we now define a few constants depending on $\textbf{k}$ and $\varphi$ that provide us with the upper and lower bounds for $\rho(\mathbb{B}_{\textbf{k},\varphi})$.

\begin{defn}\label{defn:upper_bound}
Let  $\alpha = \textbf{k}-\varphi+1$,  $\beta=\sqrt{\textbf{k}- 1} $ and 
$$\begin{array}{l l l}
			a_{1}= 2 \alpha,       & b_{1}=\alpha^2 +\alpha -\textbf{k}, & c_{1}= \alpha   -\textbf{k} +1, \\
			\\
			a_{2}=3\alpha^2 +\alpha -\textbf{k}, & b_{2}=\alpha^3 +\alpha^2-(\alpha+1)\beta^2+1, & c_{2}=\alpha^2 -\alpha \beta^2 +1, \\
			\\
			a_{3}=3\beta -\alpha, & b_{3}=2\beta(\beta-\alpha)+ (\alpha -1), & c_{3}=(\alpha -\beta -1 )\beta + (2-\alpha)\alpha, \\
			\\
			a_{4}=(5\beta-3\alpha)\beta+\alpha-1, & b_{4}=(\alpha-\beta-1)\beta & c_{4}=(\alpha-\beta-1)\beta^2\\
			                                      & \qquad +(2-\alpha)\alpha +1,  & \qquad + ((2-\alpha)\alpha +1)\beta +(1-\alpha).
\end{array}$$
For $2 < \varphi \leq  \textbf{k}-\sqrt{k-1}+1$,  let
\begin{align*}
 U_{\textbf{k}, \varphi}=
  \begin{cases}
 \alpha +\dfrac{-b_{1} + \sqrt{b_{1}^2 - 4a_{1}c_{1}} }{2a_{1}}, & \textup{ if } \varphi \textup{ is even,}\\ 
 \\
   \alpha +\dfrac{-b_{2} + \sqrt{b_{2}^2 - 4a_{2}c_{2}} }{2a_{2}}, &  \textup{ if }  \varphi  \textup{ is odd. }
  \end{cases}
  \end{align*}
For $ \textbf{k}-\sqrt{k-1}+1 < \varphi \leq  \textbf{k}-1$,  let
  \begin{align*}
 U_{\textbf{k}, \varphi}=
  \begin{cases}
 \beta +\dfrac{-b_{3} + \sqrt{b_{3}^2 - 4a_{3}c_{3}} }{2a_{3}}, & \textup{ if } \varphi \textup{ is even,}\\ 
 \\
  \beta +\dfrac{-b_{4} + \sqrt{b_{4}^2 - 4a_{4}c_{4}} }{2a_{4}}, &  \textup{ if }  \varphi  \textup{ is odd. }
  \end{cases}
  \end{align*}
\end{defn}

\begin{defn}\label{defn:lower_bound}
Let
  \begin{align*}
 L_{\textbf{k}, \varphi}=
  \begin{cases}
  \dfrac{1+\sqrt{1+4\varphi}}{2}, & \textup{ if } \varphi \textup{ is even,}\\ 
 \\
  \dfrac{1+\sqrt{1+4(\varphi-1)}}{2}, &  \textup{ if }  \varphi  \textup{ is odd. }
  \end{cases}
  \end{align*}

\end{defn}

\begin{theorem}
For $\varphi >2$, let $\rho(\mathbb{B}_{\textbf{k},\varphi})$ be the spectral radius of $\mathbb{B}_{\textbf{k},\varphi}$. Then,
$$ L_{\textbf{k}, \varphi} \leq \rho(\mathbb{B}_{\textbf{k},\varphi}) <  U_{\textbf{k}, \varphi}. $$
Moreover, for the lower bound equality attains if and only if $\varphi=\textbf{k}-1.$
\end{theorem}

\begin{proof}
We first proceed with the proof for the lower bound. If $\varphi$ is even, then the friendship graph $F_{\varphi+1}$ is subgraph of $\mathbb{B}_{\textbf{k},\varphi}$, and it is proper subgraph unless  $\varphi = \textbf{k}-1$. Therefore, using Lemma~\ref{lem:sr_edge} and Corollary~\ref{cor:friend}, we  have $\dfrac{1+\sqrt{1+4\varphi}}{2} = \rho(F_{\varphi+1}) \leq  \rho(\mathbb{B}_{\textbf{k},\varphi}),$ and equality attains if and only if $\varphi = \textbf{k}-1$.  If $\varphi$ is odd, then $F_\varphi$ is a proper subgraph of $\mathbb{B}_{\textbf{k},\varphi}$.  By Lemma~\ref{lem:sr_edge} and Corollary~\ref{cor:friend}, we  get $\dfrac{1+\sqrt{1+4(\varphi-1)}}{2} = \rho(F_{\varphi}) \leq  \rho(\mathbb{B}_{\textbf{k},\varphi}).$ From the above arguments and Definition~\ref{defn:lower_bound}, we obtained the desired inequality for the lower bound. Next, we consider the following cases to prove the result for the upper bound.\\

\noindent {\bf \underline{Case $1$}:} $\varphi$ is even.\\

Let $\alpha = \textbf{k}-\varphi+1$ and $\beta=\sqrt{\textbf{k}- 1} $. By Eqn.~\eqref{eqn:alpha_beta1}, we have $\rho(\mathbb{B}_{\textbf{k}, \varphi}) > \max\{\alpha, \beta\}$. Using Lemma~\ref{lem:roots}, $\rho(\mathbb{B}_{\textbf{k},\varphi})$ is the largest root of the polynomial $f(x)$, where 
\begin{align*}
f(x)&=x^3-(\textbf{k}-\varphi+1)x^2-(\varphi-1)x+(\varphi-1)(\textbf{k}+\varphi)+1\\
    &=x^3 -\alpha x^2 + (\alpha - \textbf{k})x - (\alpha^2 - (\textbf{k} +1 )\alpha + \beta^2).
\end{align*}

 We now consider the subcases $\alpha \leq \beta$ and $\beta < \alpha$ separately to complete the proof.  In view of $2< \varphi \leq \textbf{k}-1$,  $\beta \leq \alpha$ is equivalent to $2 < \varphi \leq  \textbf{k}-\sqrt{k-1}+1$  and  $\alpha < \beta$ is equivalent to  $ \textbf{k}-\sqrt{k-1}+1 < \varphi \leq  \textbf{k}-1$.\\

\noindent {\bf \underline{Subcase $1.1$}:}  $2 < \varphi \leq  \textbf{k}-\sqrt{k-1}+1$ ({\it{i.e.,} } \textup{for} $\beta \leq \alpha$).\\

Let $\rho(\mathbb{B}_{\textbf{k},\varphi})= \alpha +\delta.$  Then, $\delta> 0$ and   $\delta$ is the largest  root of
\begin{equation}\label{eqn:f_alpha}
f(\alpha + x)= x^3 +2\alpha x^2 + (\alpha^2 +\alpha -\textbf{k})x +(\alpha   -\textbf{k} +1).
\end{equation}
Therefore, using $\delta> 0$, we have $2\alpha \delta^2 + (\alpha^2 +\alpha -\textbf{k})\delta +(\alpha   -\textbf{k} +1) < 0.$ Using $\alpha \geq 2$  and $\varphi >2$, we have $\alpha > 0$ and $\alpha -  \textbf{k} +1= -(\varphi -2) <0$. Therefore, by Lemma~\ref{lem:inteq} and Definition~\ref{defn:upper_bound}, we get $\rho(\mathbb{B}_{\textbf{k},\varphi})<  U_{\textbf{k}, \varphi}$.\\

\noindent {\bf \underline{Subcase $1.2$}:} $ \textbf{k}-\sqrt{k-1}+1 < \varphi \leq  \textbf{k}-1$ ({\it{i.e.,}}   \textup{for} $\alpha < \beta$).\\

Let $\rho(\mathbb{B}_{\textbf{k},\varphi})= \beta +\delta.$  Then, $\delta> 0$ and   $\delta$ is the largest  root of 
$$f(\beta + x)= x^3 + (3\beta -\alpha)x^2 + (3\beta^2-2\alpha\beta + (\alpha -\textbf{k}))x + [\beta^3 - \alpha \beta^2 + (\alpha -\textbf{k})\beta  - (\alpha^2 - (\textbf{k} +1 )\alpha + \beta^2) ].$$
Note that, $\beta^3 - \alpha \beta^2 + (\alpha -\textbf{k})\beta  - (\alpha^2 - (\textbf{k} +1 )\alpha + \beta^2)= (\beta^2 + \alpha -\textbf{k} -\beta )\beta + ( (\textbf{k}+1)-\beta^2-\alpha)\alpha = (\alpha -\beta -1 )\beta + (2-\alpha)\alpha $, and 
$3\beta^2-2\alpha\beta + (\alpha -\textbf{k})= 3\beta^2-2\alpha\beta + \alpha - \beta^2-1 =2\beta(\beta-\alpha)+ (\alpha -1)$.
%%\begin{align*}
%%& \beta^3 - \alpha \beta^2 + (\alpha -\textbf{k})\beta  - (\alpha^2 - (\textbf{k} +1 )\alpha + \beta^2) \\
%%= & (\beta^2 + \alpha -\textbf{k} -\beta )\beta + ( (\textbf{k}+1)-\beta^2-\alpha)\alpha\\
%%= &(\alpha -\beta -1 )\beta + (2-\alpha)\alpha.
%%\end{align*}
Therefore, using $\delta> 0$ and  $\delta$ is the largest  root of 
\begin{equation}\label{eqn:f_beta}
f(\beta + x)=  x^3 + (3\beta -\alpha)x^2 + (2\beta(\beta-\alpha)+ (\alpha -1))x + [(\alpha -\beta -1 )\beta + (2-\alpha)\alpha],
\end{equation}
%%$$f(\beta + x)=  x^3 + (3\beta -\alpha)x^2 + (2\beta(\beta-\alpha)+ (\alpha -1))x + [(\alpha -\beta -1 )\beta + (2-\alpha)\alpha],$$ 
and hence
$ (3\beta -\alpha)\delta^2 + (2\beta(\beta-\alpha)+ (\alpha -1))\delta + [(\alpha -\beta -1 )\beta + (2-\alpha)\alpha] <0.$  Using  $\beta >\alpha  \geq 2$, we have  $(3\beta -\alpha)>0$ and $[(\alpha -\beta -1 )\beta + (2-\alpha)\alpha] <0.$ Hence,  using Lemma~\ref{lem:inteq} and Definition~\ref{defn:upper_bound}, the desired inequality follows.\\

\noindent {\bf \underline{Case $2$}:} $\varphi$ is odd.\\

Similar to Case~$1$, we consider  $\alpha \leq \beta$ and $\beta < \alpha$ as separate subcases. Using Lemma~\ref{lem:roots}, it is easy to see that $\rho(\mathbb{B}_{\textbf{k},\varphi})$ is the largest root of the polynomial $g(x)=x [f(x) + 1]+(1-\alpha).$\\

\noindent {\bf \underline{Subcase $2.1$}:}  $2 < \varphi \leq  \textbf{k}-\sqrt{k-1}+1$ ({\it{i.e.,} } \textup{for} $\beta \leq \alpha$).\\

Let $\rho(\mathbb{B}_{\textbf{k},\varphi})= \alpha +\delta.$  Then, $\delta> 0$ and   $\delta$ is the largest  root of $g(\alpha+x)=(\alpha+ x) [f(\alpha+ x) + 1]+(1-\alpha).$ Using Eqn.~\eqref{eqn:f_alpha}, we have
\begin{align*}
g(\alpha+x) &= x^4 +3\alpha x^3+ (3\alpha^2 +\alpha- \textbf{k})x^2+[\alpha^3 +\alpha^2-(\textbf{k}-1)(\alpha+1)+1]x +(\alpha^2 -(\textbf{k}-1)\alpha +1)\\
&= x^4 +3\alpha x^3+ (3\alpha^2 +\alpha- \textbf{k})x^2+[\alpha^3 +\alpha^2-(\alpha+1)\beta^2+1]x +(\alpha^2 -\alpha \beta^2 +1).
\end{align*}
%%$$g(\alpha+x)= x^4 +3\alpha x^3+ (3\alpha^2 +\alpha \textbf{k})+[\alpha^3 +\alpha^2-(\textbf{k}-1)(\alpha+1)+1]x +(\alpha^2 -(\textbf{k}-1)\alpha +1)$$
Using $\alpha\geq 2$ and $\delta> 0$, we have $\delta^4 +3\alpha \delta^3>0$. Thus, $$(3\alpha^2 +\alpha -\textbf{k})\delta^2+[\alpha^3 +\alpha^2-(\alpha+1)\beta^2+1]\delta +(\alpha^2 -\alpha \beta^2 +1)< 0.$$ 
Now, using $\beta \leq \alpha$ and $\alpha \geq 2$, we get $\textbf{k}-1 = \beta^2 \leq \alpha^2 $  which implies that $1\leq \alpha -1 \leq \alpha^2 +\alpha -\textbf{k}.$ Hence $(3\alpha^2 +\alpha -\textbf{k}) > 0$. Further, using $\alpha$ is a  positive integer and $2\leq \alpha < \textbf{k}-1$, gives $\alpha^2 -\alpha \beta^2 = \alpha (\alpha -(\textbf{k}-1)) \leq -2.$ Hence $(\alpha^2 -\alpha \beta^2 +1)< 0.$ Therefore,  Lemma~\ref{lem:inteq} and Definition~\ref{defn:upper_bound} gives the desired bound.\\

\noindent {\bf \underline{Subcase $2.2$}:} $ \textbf{k}-\sqrt{k-1}+1 < \varphi \leq  \textbf{k}-1$ ({\it{i.e.,}}   \textup{for} $\alpha < \beta$).\\

Let $\rho(\mathbb{B}_{\textbf{k},\varphi})= \beta +\delta.$  Then, $\delta> 0$ and   $\delta$ is the largest  root of $g(\beta+x)=(\beta+ x) [f(\beta+ x) + 1]+(1-\alpha).$ Using Eqn.~\eqref{eqn:f_beta}, we have

\begin{align*}
g(\beta + x)=& x^4 + (4\beta-\alpha)x^3 + [(5\beta -3\alpha)\beta+(\alpha -1)]x^2 \\ 
             & \qquad  + [(\alpha-\beta-1)\beta+(2-\alpha)\alpha +1]x + [(\alpha-\beta-1)\beta^2+ ((2-\alpha)\alpha +1)\beta +(1-\alpha).
\end{align*}

Using $\delta> 0$ and $\beta >\alpha$, we have $\delta^4 + (4\beta-\alpha)\delta^3 >0$ and hence 
$$ [(5\beta -3\alpha)\beta+(\alpha -1)]\delta^2 
              + [(\alpha-\beta-1)\beta+(2-\alpha)\alpha +1]\delta + [(\alpha-\beta-1)\beta^2+ ((2-\alpha)\alpha +1)\beta +(1-\alpha) <0.$$
Using  $\beta > \alpha \geq 2$, we get $(5\beta -3\alpha)\beta+(\alpha -1)>0.$ Further, $\varphi >2$ implies that $\textbf{k}>2$ and hence using $\textbf{k}>2$ and  $\beta > \alpha \geq 2$, we have
$$
(\alpha-\beta-1)\beta^2+ ((2-\alpha)\alpha +1)\beta +(1-\alpha)=(\alpha-\beta)\beta^2+ (2-\alpha)\alpha\beta +(1-\beta )\beta  +(1-\alpha)<0.$$
Therefore, given  Lemma~\ref{lem:inteq} and Definition~\ref{defn:upper_bound}, the desired inequality follows. This completes the proof.
\end{proof}

\section{Conclusion}

In this manuscript, we consider the spectral radius of graphs in the class of block graphs $\mathbf{Bl}(\textbf{k}, \varphi)$ with a fixed number of vertices $\textbf{k}$ and a  given dissociation number $\varphi$. We first prove a few elementary results on the adjacency matrix of a graph, the spectral radius and a corresponding Perron vector. Using these results, we show a graph with maximal spectral radius in $\mathbf{Bl}(\textbf{k}, \varphi)$ satisfies a few necessary properties. These properties guarantee the existence and uniqueness of the block graph $\mathbb{B}_{\textbf{k},\varphi}$  that maximize the spectral radius in  $\mathbf{Bl}(\textbf{k}, \varphi)$. Finally, we obtain bounds on the spectral radius of  $\mathbb{B}_{\textbf{k},\varphi}$.\\

\noindent{ \textbf{\Large Acknowledgements}}: Joyentanuj Das is partially supported by the National Science and Technology Council in Taiwan (Grant ID: NSTC-111-2628-M-110-002).

\small{
}


\begin{thebibliography}{20}

\bibitem{Bapat}  R.B. Bapat,  Graphs and matrices. Second Edition, Hindustan Book Agency, New Delhi, (2014).

\bibitem{Biggs}  N.  Biggs, 
                 Algebraic Graph Theory. Second Edition.,
                 Cambridge Mathematical Library.  
                 Cambridge University Press (1993).

\bibitem{Bresar} B. Brešar, F. Kardoš, J. Katrenič, G. Semanišin, 
                 Minimum $k$-path vertex cover, 
                 Discrete Appl. Math. 159 (12) (2011),  1189-1195.

\bibitem{Bresar1} B. Brešar, M. Jakovac, J. Katrenič, G. Semanišin, A. Taranenko,  
                  On the vertex k-path cover. 
                  Discrete Appl. Math. 161 (13-14) (2013), 1943-1949.

\bibitem{Brualdi} R.A. Brualdi, E.S. Solheid, 
                  On the spectral radius of complementary acyclic matrices of zeros and ones, 
                  SIAM J. Algebraic Discrete Methods, 7(2) (1986) 265-272.

\bibitem{Conde}  C. M. Conde, E. Dratman, L. N.Grippo, 
                 On the spectral radius of block graphs with prescribed independence number $\alpha$.         
                 Linear Algebra Appl., 614 (2021), 111-124.

\bibitem{Das} J. Das, S. Mohanty, 
              On the spectral radius of bi-block graphs with given independence number $\alpha$, 
              Appl. Math. Comput. 402 (2021) 125912, 8 pp.
              
\bibitem{KCD} K.C. Das, P. Kumar, 
Some new bounds on the spectral radius of graphs. Discrete Math. 281 (2004), no. 1-3, 149–161.              

\bibitem{Feng} L. Feng, G. Yu, 
               Spectral radius of unicyclic graphs with given independence number, 
               Util. Math., 84 (2011),  33-43.

\bibitem{Guo} J. Guo, J. Shao, 
              On the spectral radius of trees with fixed diameter, 
              Linear Algebra Appl., 413 (1) (2006),  131-147.

\bibitem{Chu} C. Ji, M. Lu, 
              On the spectral radius of trees with given independence number, 
              Linear Algebra Appl., 488 (2016),  102-108.

\bibitem{Li} Q. Li, K.Q. Feng, 
             On the largest eigenvalue of a graph, 
             (Chinese), Acta Math. Appl. Sin . 2 (1979) 167–175.

\bibitem{Lui} H. Liu, M. Lu, F. Tian, 
              On the spectral radius of graphs with cut edges, 
              Linear Algebra Appl., 389 (2004),  139-145.

\bibitem{Lu} H. Lu, Y. Lin, 
             Maximum spectral radius of graphs with given connectivity, 
             minimum degree and independence number, 
             J. Discrete Algorithms, 31 (2015),  113-119.

\bibitem{Lou} Z. Lou, J. Guo, 
              The spectral radius of graphs with given independence number, 
              Discrete Math. 345 (4) (2022) 112778, 13 pp.

\bibitem{Orlovich} Y. Orlovich, A. Dolgui, G. Finke, V. Gordon, F. Werner, 
                   The complexity of dissociation set problems in graphs, 
                   Discrete Appl. Math. 159 (13) (2011) 1352-1366.

\bibitem{DS} D. Stevanović, 
             Spectral Radius of Graphs, 
             Academic Press, 2014.

\bibitem{Tu} J. Tu, Z. Zhang, Y. Shi, 
             The maximum number of maximum dissociation sets in trees, 
             J. Graph Theory 96 (4) (2021) 472-489.

\bibitem{Tu1} J. Tu, Y. Li, J. Du, 
              Maximal and maximum dissociation sets in general and triangle-free graphs. 
              Appl. Math. Comput. 426 (2022), 127107, 11 pp.

\bibitem{Xiao} B. Wu, E. Xiao, Y. Hong, 
               The spectral radius of trees on $k$ pendant vertices, 
               Linear Algebra Appl., 395 (2005), 343-349.

\bibitem{Xu} M. Xu, Y. Hong, J. Shu, M. Zhai, 
             The minimum spectral radius of graphs with a given independence number, 
             Linear Algebra Appl., 431 (5-7) (2009), 937-945.
\bibitem{You} L. You, M. Yang, W. So, W. Xi, 
			  On the spectrum of an equitable quotient matrix and its application. Linear Algebra Appl., 577, (2019), 21-40.              
             
\bibitem{Yuan} H. Yuan, A bound on the spectral radius of graphs.
Linear Algebra Appl. 108 (1988), 135-139.

\end{thebibliography}
\end{document}